\documentclass[12pt]{amsart}

\usepackage{graphicx}
\usepackage{amssymb}
\usepackage{float}
\usepackage{fullpage}

\newtheorem{theorem}{Theorem}[section]
\newtheorem{lemma}[theorem]{Lemma}
\newtheorem{corollary}[theorem]{Corollary}
\newtheorem{proposition}[theorem]{Proposition}

\theoremstyle{definition}
\newtheorem{definition}[theorem]{Definition}

\theoremstyle{remark}
\newtheorem{remark}[theorem]{Remark}

\numberwithin{equation}{section}

\usepackage{physics}
\usepackage{multicol}

\usepackage{comment}

\usepackage{mathtools}

\newcommand{\Z}{{\mathbb Z}}

\newcommand{\Mod}[1]{\ (\mathrm{mod}\ #1)}

\renewcommand{\phi}{{\varphi}}
\renewcommand{\le}{{\,\leqslant}}
\renewcommand{\ge}{{\,\geqslant}}

\newcommand{\1}{{\mathbf 1}}

\newcommand{\eps}{{\varepsilon}}

\begin{document}

\title[Primes in arithmetic progressions \& Shifted primes]
{Primes in arithmetic progressions to large moduli\\
and shifted primes without large prime factors
}

\author{Jared Duker Lichtman}
\address{Mathematical Institute, University of Oxford, Oxford, OX2 6GG, UK}

\email{jared.d.lichtman@gmail.com}



\date{November 4, 2022.}


\begin{abstract}
We prove the infinitude of shifted primes $p-1$ without prime factors above $p^{0.2844}$. This refines $p^{0.2961}$ from Baker and Harman in 1998. Consequently, we obtain an improved lower bound on the the distribution of Carmichael numbers.

Our main technical result is a new mean value theorem for primes in arithmetic progressions to large moduli. Namely, we estimate primes of size $x$ with quadrilinear forms of moduli up to $x^{17/32}$. This extends moduli beyond $x^{11/21}$, recently obtained by Maynard, improving $x^{29/56}$ from well-known 1986 work of Bombieri, Friedlander, and Iwaniec.
\end{abstract}

\maketitle


\section{Introduction}\label{sec:Introduction}
%
%
%
%
%
%
%
%

Let $P^+(n)$ denote the largest prime factor of an integer $n>1$. Following an old conjecture of Erd\H{o}s \cite{Erdosnormal, Erdospseudo}, we expect there are infinitely many primes $p$ with $P^+(p-a)\le\, p^\eps$, for any $\eps>0$. We prove the infinitude of primes $p$ with $P^+(p-a)\le\, p^{0.2844}$, from the following quantitative result.

\begin{theorem}\label{thm:shftdprime}
For fixed nonzero $a\in\Z$ and $\beta > 15/32\sqrt{e}= 0.2843\cdots$, there exists $C\ge1$ such that
\begin{align}\label{eq:shftdprime}
\sum_{\substack{x<p\le 2x\\P^+(p-a) \le x^\beta}}1 \ \gg \ \frac{x}{(\log x)^C}.
\end{align}
\end{theorem}

The exponent $0.2844$ in Theorem \ref{thm:shftdprime} gives a roughly $4\%$ refinement over the previous record exponent $0.2961$ of Baker and Harman \cite{BH}. The table below gives a chronology of the known lower bounds on $\beta$.

\vspace{1.5mm}

\begin{center}
\begin{tabular}{c|l|l}
Year & Author(s) & $\beta>$ \\
\hline    
1998 & Baker--Harman \cite{BH} & 0.2961\\
1989 & Friedlander \cite{Fried} & $0.3032\cdots=1/2\sqrt{e}$\\
1986 & Fouvry--Grupp \cite{FG} & $0.3174\cdots=3/7e^{.3}$\\
1983 & Balog \cite{Balog} & 0.35\\
1980 & Pomerance \cite{Pomer} & $0.4490\cdots=625/512e$\\
1979 & Wooldridge \cite{Wool} & $0.8284\cdots=2(\sqrt{2}-1)$\\
1935 & Erd\H{o}s \cite{Erdosnormal} & $1-\delta$ \ for some $\delta>0$
\end{tabular}
\end{center}
\vspace{1.mm}

In recent decades, this problem has attracted increased attention,
due in part to applications to cryptography (see \cite{CrandPomer} \cite{Granvillesmooth} \cite{PomerShpar} for further discussion).
Moreover, this problem is of independent interest as it sheds light on the subtle interplay between 
addition and multiplication. Indeed, long before connections to cryptography, in 1935 Erd\H{o}s \cite{Erdosnormal} showed the existence of some $\delta>0$ for which infinitely many primes $p$ satisfy $P^+(p-1)\le\, p^{1-\delta}$.

Theorem \ref{thm:shftdprime} implies the following lower bound on the distribution of Carmichael numbers. Recall a composite number $n$ is a Carmichael number if $n$ is a pseudoprime to every base $b$, that is, $b^{n-1}\equiv 1\Mod{n}$ for all coprime $(b,n)=1$.

\begin{corollary}\label{cor:Carmichael}
There are at least $x^{0.3389}$ Carmichael numbers up to $x$, sufficiently large.
\end{corollary}

The infinitude of Carmichael numbers we first proven in the landmark result of Alford--Granville--Pomerance \cite{AGP}. Moreover, their argument gave the quantitative lower bound $x^{\frac{5}{12}(1-\beta)}$, for $\beta>0$ satisfying \eqref{eq:shftdprime}. The current record is $x^{0.4736(1-\beta)}$ due to Harman \cite{HarmWatt}.
Corollary \ref{cor:Carmichael} then follows by combining this with Theorem \ref{thm:shftdprime}. Note the exponent $.4736(1-15/32\sqrt{e}) = .3389\cdots$ in Corollary \ref{cor:Carmichael} improves on $.4736(1-.2961)= .3333\cdots$ from Harman \cite{HarmWatt}.

In addition, we deduce the following consequence on the distribution of values of the Euler $\phi$ function.

\begin{corollary}\label{cor:Eulerphi}
Denote by $m_1<m_2<\cdots$ the integers $m\in\Z$ for which $m=\phi(n)$ admits at least $m^{0.7156}$ solutions $n\in\Z$. Then the sequence $(m_i)$ is infinite, and satisfies
\begin{align*}
\lim_{i\to\infty} \frac{\log m_{i+1}}{\log m_i} = 1.
\end{align*}
\end{corollary}

Corollary \ref{cor:Eulerphi} follows directly by Theorem \ref{thm:shftdprime}, from the well-established method of Erd\H{o}s and Pomerance \cite{Erdospseudo, Pomer}. Note the exponent $1-15/32\sqrt{e}=.7156\cdots$ in Corollary \ref{cor:Eulerphi} improves on $1-.2961=.7039$ from Harman \cite{HarmWatt}.

\subsection{Primes in arithmetic progressions to large moduli} 

As the main technical result of the article, we establish a new estimate for primes in arithmetic progressions with quadrilinear forms of moduli. Let $\tau(q)$ denote the divisor function, and $\pi(x;q,a)$ the count of primes up to $x$ congruent to $a$ (mod $q$).
%
%
%
%
%
%
%
%
\begin{theorem}\label{thrm:trilinear}
Fix nonzero $a\in\mathbb{Z}$. Let $\eps>0$ and let $Q,R,S$ satisfy
\begin{align}\label{eq:Cons0}
QR < x^{1/2+\eps}, \qquad  QS^2 < x^{1/2-2\eps}\qquad
S^2 < R < x^{1/32-\eps}.
\end{align}
Let $\lambda_q,\nu_q,\eta_q,\mu_q$ be complex sequences with $,|\lambda_{q}|,|\nu_q|,|\eta_{q}|,|\mu_{q}|\le \tau(q)^{B_0}$.
Then for every $A>0$ we have
\[
\underset{(qrst,a)=1}{\sum_{q\le Q}\sum_{r\le  R}\sum_{s\le S}\sum_{t\le S}} \lambda_{q}\nu_{r}\eta_{s}\mu_t\Bigl(\pi(x;qrst,a)-\frac{\pi(x)}{\phi(qrst)}\Bigr)\ll_{a,\eps,A}\frac{x}{(\log{x})^A}.
\]
\end{theorem}
%
%
%
%
%
%
%
Theorem \ref{thrm:trilinear} may handle quadrilinear forms of moduli up to $x^{17/32-\eps}$, for the choices $(Q,R,S)= (x^{15/32+2\eps},x^{1/32-\eps},x^{1/64-2\eps})$. Specifically, for the application to our main Theorem \ref{thm:shftdprime} we use
\begin{align}\label{eq:quadrilinear}
\underset{(qrst,a)=1}{\sum_{q\le  x^{15/32+2\eps}}\sum_{r\le  x^{1/32-\eps}}\sum_{s\le  x^{1/64-2\eps}}\sum_{t\le  x^{1/64-2\eps}}} \lambda_{q}\nu_{r}\eta_{s}\mu_t\Bigl(\pi(x;qrst,a)-\frac{\pi(x)}{\phi(qrst)}\Bigr)\ll_{a,\eps,A}\frac{x}{(\log{x})^A}.
\end{align}

Previously, the strongest results for these applications had involved bilinear forms. In their celebrated 1986 work, Bombieri, Friedlander, and Iwaniec \cite[Theorem 8]{BFI1} handled bilinear forms of moduli up to $x^{29/56}$, showing
\begin{align}\label{eq:bilinear}
\underset{(qr,a)=1}{\sum_{q\le  Q}\sum_{r\le  R}} \lambda_{q}\nu_{r}\Bigl(\pi(x;qr,a)-\frac{\pi(x)}{\phi(qr)}\Bigr)\ll_{a,\eps,A}\frac{x}{(\log{x})^A}
\end{align}
with $Q<x^{1/3}$, $R<x^{1/5}$, $Q^5R^2<x^2$, and $QR<x^{29/56}$. This bilinear estimate was only recently extended to moduli up to $x^{11/21-\eps}$ by Maynard \cite[Theorem 1.1]{JM1}, deducing \eqref{eq:bilinear} with $(Q,R)=(x^{10/21},x^{1/21-\eps})$ (even with weights replaced by absolute values) Note $29/56= .5178\cdots$, $11/21= .5238\cdots$, and $17/32= 0.5313\cdots$.

%
%
%
%
%
%
%
%

In context, Theorem \ref{thrm:trilinear} may be viewed as an interpolating result between \cite[Theorem 1.1]{JM1} and \cite[Theorem 1.1]{JM2}. Namely, the quadrilinear weights in Theorem \ref{thrm:trilinear} are more flexible than the absolute values appearing in \cite[Theorem 1.1]{JM1}, but appear more rigid than the `triply well-factorable' weights in \cite[Theorem 1.1]{JM2}. As such, the size of moduli $x^{17/32-\epsilon}$ in Theorem \ref{thrm:trilinear} exceeds $x^{11/21-\epsilon}$ in \cite[Theorem 1.1]{JM1}, while not to the full extent of $x^{3/5-\epsilon}$ in \cite[Theorem 1.1]{JM2}. The balance of flexibility and strength in Theorem \ref{thrm:trilinear} is molded to our application to Theorem \ref{thm:shftdprime}. Whereas, the shifted primes problem appears too rigid for \cite[Theorem 1.1]{JM2} to be applicable.

%
%
%
%
%
%
%
%

\section{Proof outline}

In this section we outline the proofs of Theorems \ref{thm:shftdprime} and \ref{thrm:trilinear}.

Following the previous method of Baker--Harman \cite{BH}, we restrict our attention to special factorizations $p-a=lp_0mn$ with $l\sim x^{2\theta-1}$ smooth, $p_0$ prime, and $p_0m,\,n\sim x^{1-\theta}$. We aim to take $\theta>1/2$ as large as possible. An inclusion-exclusion argument then reduces the problem to estimating primes in arithmetic progressions of the form
\begin{align}\label{eq:APoutline}
\sum_{l} \sum_{p_0}\sum_{m} \Big(\pi(x;a,lp_0m) - \frac{\pi(x)}{\phi(lp_0m)} \Big).
\end{align}

The desired bound $x/(\log x)^A$ for \eqref{eq:APoutline} above follows from \cite[Theorem 1.1]{JM1} when $\theta=11/21$, by grouping $p_0$,$m$ together and inserting absolute values (this would already give some improvement over Baker--Harman). We can hope to do better by exploiting the fact we don't need absolute values. That is, the argument for \cite[Theorem 1.1]{JM1} was limited by $\theta\le 11/21$ only for very specific terms in the decomposition of primes. Namely, when one factor in the decomposition is of size $x^{1/7}$, the `Fouvry-style' estimate \cite[Proposition 12.1]{JM1} break down. When the coefficients factor as $c_{q,r}=\lambda_q\nu_r$, we may use a stronger result of Iwaniec--Pomykala \cite{IwaniecPomykala} in this critical case. 

In addition, we strengthen the `Zhang-style' estimate \cite[Proposition 8.2]{JM1} in the case of trilinear forms of moduli. Lastly, to combine these `Fouvry-style' and `Zhang-style' trilinear estimates in a compatible manner, we take a common refinement of the corresponding systems of conditions. This leads to the single system \eqref{eq:Cons0} for the final quadrilinear estimate. 
When combined with other estimates of Maynard \cite{JM1}, this establishes the main Theorem \ref{thrm:trilinear} with $\theta=17/32$.

Roughly speaking, the Zhang-style argument handles sums of the form
\begin{align*}
\sum_{q}\sum_{r} c_{q,r}\sum_{n}\alpha_n\sum_{m}\beta_m \Big(\1_{nm\equiv a (qr)} - \frac1{\phi(qr)}\Big).
\end{align*}
We apply Cauchy-Schwarz in the $q,m$ variables (smoothing the $m$-summation), use Poission summation in the $m$ variable, and simplify the exponential sums to roughly give
\begin{align*}
\sum_{q}\sum_{r_1,r_2}c_{q,r_1}c_{q,r_2}\sum_{n_1,n_2}\alpha_{n_1}\alpha_{n_2}\sum_{h<QR^2N/x}e\Big(a h\Big( \frac{\overline{n_1r_2}}{q r_1} + \frac{\overline{n_2r_1q}}{r_2} \Big)\Big).
\end{align*}

Then we want to apply Cauchy-Schwarz in the $q,n_1,n_2$ variables to eliminate the unknown coefficients $\alpha_n$. Then applying Poisson summation in the $n_1,n_2$ variables (inserting a smooth majorant) leads to sums of Kloosterman sums, which may be handled by the Weil bound.

In our case of trilinear forms of moduli, we may restrict to those $r$ which factor and consider the coefficients $c_{q,rs}=\lambda_q\nu_r\eta_s$. In this situation, we may also apply Cauchy-Schwarz in in the $s_1,s_2$ variables. This new situation has the significant benefit of fewer off-diagonal terms, and the remaining diagonal terms have less sparse summations in auxiliary variables. This allows for improved estimates in this critical case, and leads to $\theta=17/32$ when combined with the other estimates.

\begin{remark}
The previous argument of Baker--Harman \cite{BH} proceeds by estimating primes in arithmetic progressions of the form
\begin{align}\label{eq:BHoutline}
\sum_{l} \sum_{p_0}\sum_{m} \pi(x;a,lp_0m)
\end{align}
in two regimes, using the following ingredients:

(a) a mean value theorem of Bombieri--Friedlander--Iwaniec \cite[Theorem 9]{BFI1} to get asymptotics for \eqref{eq:BHoutline} in the range $x^{\beta}< p_0 < x^{1/3}$.

(b) a Harman's sieve argument to get weaker bounds for \eqref{eq:BHoutline} in the range $x^{1/3} \le p_0 \le 2x^{1-\theta}$, by focusing on the most amenable parts of the decomposition of the primes.

Now using Theorem \ref{thrm:trilinear} we upgrade (a) above, which
already achieves a superior result for shifted primes without appealing to part (b).
In principle, one could obtain some further improvements using a similar argument as in (b). In the interest of clarity of presentation we do not pursue this.
\end{remark}

\section{Shifted primes without large prime factors}

In this section we deduce Theorem \ref{thm:shftdprime} from Theorem \ref{thrm:trilinear}.

Let $\theta=17/32-\eps$ and take $\beta>(1-\theta)/\sqrt{e}$. Define the integer $H = \lceil (2\theta-1)/\eps\rceil$. In particular $\frac{2\theta-1}{H}\in[\eps/2,\eps]$. Denote the dyadic subset $L = \{l\sim x^{\frac{2\theta-1}{H}} : (l,a)=1\}$ and define the sequence $\mathcal G = \{l_1\cdots l_H : l_i\in L\}$, so that $l\asymp x^{2\theta-1}$ for all $l\in \mathcal G$, and
\begin{align*}
|\mathcal G|  = |L|^H \gg x^{2\theta-1}
\end{align*}

Consider
\begin{align}
\mathcal N \ & = \ \big\{(p,l,m,n) \;:\; p-a = lmn, \ l\in\mathcal G, \, m,n\sim x^{1-\theta},\, (m,a)=1,\, p\sim x \big\}\\
\mathcal N' \ & = \ \big\{(p,l,m,n)\in \mathcal N \;:\; P^+(p-a) \le x^\beta \big\}
\end{align}
Observe that it suffices to prove
\begin{align}\label{eq:N1bound}
|\mathcal N'| \gg \frac{x}{\log x}.
\end{align}
Indeed, letting $\Gamma(p)=|\{(l,m,n) \,:\, (p,l,m,n)\in \mathcal N'\}|$, we have $|\mathcal N'| = \sum_{p\sim x}\Gamma(p)$. By Cauchy-Schwarz,
\begin{align*}
|\mathcal N'|^2 = \bigg(\sum_{p\sim x}\Gamma(p)\bigg)^2 \le \sum_{p\sim x}\Gamma(p)^2\cdot\sum_{\substack{p\sim x\\\Gamma(p)>0}}1.
\end{align*}
 Also $\Gamma(p) \le \tau(p-a)^{B_0}$ so by a divisor bound $\sum_{p\sim x}\Gamma(p)^2 \ll |\mathcal N'|(\log x)^{B}$, and so

\begin{align*}
\sum_{\substack{p\sim x\\P^+(p-a) \le x^\beta}}1 &\ge \sum_{\substack{p\sim x\\\Gamma(p)>0}}1 \gg \frac{|\mathcal N'|}{(\log x)^B} \gg \frac{x}{(\log x)^{B+1}}
\end{align*}
by \eqref{eq:N1bound}, as desired. Thus to establish Theorem \ref{thm:shftdprime}, it suffices to show \eqref{eq:N1bound}.

Next define
\begin{align*}
\mathcal N_1 &= \{(p,l,p_0m,n)\in \mathcal N\; : \; x^\beta < p_0 \le 2x^{1-\theta}\}\\
\mathcal N_2 &= \{(p,l,m,p_0n)\in \mathcal N\; : \; x^\beta < p_0 \le 2x^{1-\theta}\}
\end{align*}
and note by symmetry $|\mathcal N_1|=|\mathcal N_2|$. We have
\begin{align}\label{eq:NN1N2}
|\mathcal N'| \ge |\mathcal N| - |\mathcal N_1| - |\mathcal N_2| = |\mathcal N| - 2|\mathcal N_2|.
\end{align}

Recall by definition of $l\in\mathcal G$ we have $l\asymp x^{2\theta-1}=x^{1/16-2\eps}$ and $l=l_1\cdots l_H$ for $l_i\sim x^{(1/16-2\eps)/H}$. We therefore may split $l=l_1rst$ where $s=l_2\cdots l_{H/4-1}$, $t=l_{H/4}\cdots l_{H/2-1}$, and $r=l_{H/2}\cdots l_H$. Thus $s,t\le x^{1/64-2\eps}$ and $r\le x^{1/32-\eps}$. And letting $q = l_1mp_0 \le x^{15/32+2\eps}$, we obtain
\begin{align}\label{eq:N1coeffs}
|\mathcal N_1| &= \sum_{l\in \mathcal G}\sum_{x^{\beta}<p_0\le 2x^{1-\theta}} \sum_{\substack{m\sim x^{1-\theta}/p_0\\(m,a)=1}} 
\sum_{\substack{p\equiv a\Mod{lmp_0}\\p\sim x}}1 \nonumber\\
&= \underset{(qrst,a)=1}{\sum_{q\le  x^{15/32+2\eps}}\sum_{r\le  x^{1/32-\eps}}\sum_{s\le  x^{1/64-2\eps}}\sum_{t\le  x^{1/64-2\eps}}} \lambda_{q}\nu_{r}\eta_{s}\mu_t \sum_{\substack{p\equiv a\Mod{qrst}\\p\sim x}}1.
\end{align}
for the choice of coefficients $\lambda_q$, $\nu_r$, $\eta_s$, $\mu_t$, 
\begin{align*}
\lambda_q = \underset{q=l_1mp_0}{\sum_{l_1\sim x^{(2\theta-1)/H}}\sum_{x^{\beta}<p_0\le 2x^{1-\theta}}}1,\qquad
\nu_r =  \sum_{\substack{r=l_{H/2+1}\cdots l_H\\l_i\sim x^{(2\theta-1)/H}}}1,\\
\eta_s = \sum_{\substack{s=l_2\cdots l_{H/4}\\l_i\sim x^{(2\theta-1)/H}}}1,\qquad
\mu_t = \sum_{\substack{t=l_{H/4+1}\cdots l_{H/2}\\l_i\sim x^{(2\theta-1)/H}}}1.
\end{align*}

Now we apply the key mean value theorem in Theorem \ref{thrm:trilinear}, specifically \eqref{eq:quadrilinear}. Thus \eqref{eq:N1coeffs} becomes
\begin{align*}
\big(1+O(\mathcal L^{-A})\big)|\mathcal N_1| &= \underset{(qrst,a)=1}{\sum_{q\le  x^{15/32}}\sum_{r\le  x^{3/64}}\sum_{s\le  x^{1/64-\eps}}\sum_{t\le  x^{1/64-\eps}}} \lambda_{q}\nu_{r}\eta_{s}\mu_t\frac{\pi(x)}{\phi(qrs)}\\
&=\sum_{l\in \mathcal G}\sum_{x^{\beta}<p_0\le 2x^{1-\theta}}\sum_{\substack{m\sim x^{1-\theta}/p_0\\(m,a)=1}} \frac{\pi(x)}{\phi(lmp_0)}.
\end{align*}
Note by an elementary argument,
\begin{align*}
\sum_{\substack{m\sim M\\ (m,a)=1}}\frac{\phi(l)}{\phi(lm)} = G_l\log 2 + O\big(\tau(a)(\log x)/M\big),
\end{align*}
where $G_l = \frac{\phi(a)}{a}\prod_{p\nmid al}\Big(1+\frac{1}{p(p-1)}\Big)$. 
Note $G_l \gg_a 1$. Hence we obtain
\begin{align}\label{eq:N12}
|\mathcal N_1| &= \big(1+O(\mathcal L^{-1})\big)\,\pi(x)\log 2 \sum_{l\in \mathcal G}\frac{G_l}{\phi(l)} \sum_{x^{\beta}<p_0\le 2x^{1-\theta}}\frac{1}{p_0}.
\end{align}
By an analogous (simpler) argument, we have
\begin{align}\label{eq:N}
|\mathcal N| = \sum_{l\in \mathcal G}\sum_{\substack{m\sim x^{1-\theta}\\(m,a)=1}} \sum_{\substack{p\sim x\\p\equiv a\;(lm)}}1 = 
\big(1+O(\mathcal L^{-1})\big)\,\pi(x)\log 2 \sum_{l\in \mathcal G}\frac{G_l}{\phi(l)}
\end{align}
Indeed, this is \cite[(2.6)]{BH} which follows already by \cite[Theorem 9]{BFI1}.
In particular $|\mathcal N| \gg x/\log x$. From \eqref{eq:N12}, \eqref{eq:N} we see
\begin{align*}
|\mathcal N_1| = \big(1+O(\mathcal L^{-A})\big)|\mathcal N|\sum_{x^{\beta}<p_0\le 2x^{1-\theta}}\frac{1}{p_0} = (1+o(1))|\mathcal N|\cdot \log(\tfrac{1-\theta}{\beta}).
\end{align*}

Hence plugging back into \eqref{eq:NN1N2}, we conclude
\begin{align*}
|\mathcal N'| &\ge|\mathcal N| - 2|\mathcal N_1| \ = \ (1+o(1))|\mathcal N|\big(1-2\log(\tfrac{1-\theta}{\beta})\big) 
\gg |\mathcal N| \gg \frac{x}{\log x}
\end{align*}
by assumption $\beta > (1-\theta)/\sqrt{e}$. This gives \eqref{eq:N1bound}, and completes the proof of Theorem \ref{thm:shftdprime}.

\section{Notation}
We will use the Vinogradov $\ll$ and $\gg$ asymptotic notation, and the big oh $O(\cdot)$ and $o(\cdot)$ asymptotic notation. $f\asymp g$ will denote the conditions $f\ll g$ and $g\ll f$ both hold. Dependence on a parameter will be denoted by a subscript.

We will view $a$ (the residue class we count arithmetic functions in to different moduli $q$) as a fixed positive integer throughout the paper, and any constants implied by asymptotic notation will be allowed to depend on $a$ from this point onwards. Similarly, throughout the paper, we will let $\eps$ be a single fixed small real number; $\eps=10^{-100}$ would probably suffice. Any bounds in our asymptotic notation will also be allowed to depend on $\eps$.

The letter $p$ will always be reserved to denote a prime number. We use $\phi$ to denote the Euler totient function, $e(x):=e^{2\pi i x}$ the complex exponential, $\tau_k(n)$ the $k$-fold divisor function, $\mu(n)$ the M\"obius function. We let $P^-(n)$, $P^+(n)$ denote the smallest and largest prime factors of $n$ respectively, and $\hat{f}$ denote the Fourier transform of $f$ over $\mathbb{R}$ - i.e. $\hat{f}(\xi)=\int_{-\infty}^{\infty}f(t)e(-\xi t)dt$. We use $\mathbf{1}$ to denote the indicator function of a statement. For example,
\[
\mathbf{1}_{n\equiv a\Mod{q}}=\begin{cases}1,\qquad &\text{if }n\equiv a\Mod{q},\\
0,&\text{otherwise}.
\end{cases}
\]
We will use $(a,b)$ to denote $\gcd(a,b)$ when it does not conflict with notation for ordered pairs. For $(n,q)=1$, we will use $\overline{n}$ to denote the inverse of the integer $n$ modulo $q$; the modulus will be clear from the context. For example, we may write $e(a\overline{n}/q)$ - here $\overline{n}$ is interpreted as the integer $m\in \{0,\dots,q-1\}$ such that $m n\equiv 1\Mod{q}$. Occasionally we will also use $\overline{\lambda}$ to denote complex conjugation; the distinction of the usage should be clear from the context.  For a complex sequence $\alpha_{n_1,\dots,n_k}$, $\|\alpha\|_2$ will denote the $\ell^2$ norm $\|\alpha\|_2=(\sum_{n_1,\dots,n_k}|\alpha_{n_1,\dots,n_k}|^2)^{1/2}$.

Summations assumed to be over all positive integers unless noted otherwise. We use the notation $n\sim N$ to denote the conditions $N<n\le 2N$.

We will let $z_0:=x^{1/(\log\log{x})^3}$ and $y_0:=x^{1/\log\log{x}}$ two parameters depending on $x$, which we will think of as a large quantity. We will let $\psi_0:\mathbb{R}\rightarrow\mathbb{R}$ denote a fixed smooth function supported on $[1/2,5/2]$ which is identically equal to $1$ on the interval $[1,2]$ and satisfies the derivative bounds $\|\psi_0^{(j)}\|_\infty\ll (4^j j!)^2$ for all $j\ge 0$. (See \cite[Page 368, Corollary]{BFI2} for the construction of such a function.)
%
%
%
%
%
%
%
%
\begin{definition}[Siegel-Walfisz condition]
We say that a complex sequence $\alpha_n$ satisfies the \textbf{Siegel-Walfisz condition} if for every $d\ge 1$, $q\ge 1$ and $(a,q)=1$ and every $A>1$ we have
\begin{equation}
\Bigl|\sum_{\substack{n\sim N\\ n\equiv a\Mod{q}\\ (n,d)=1}}\alpha_n-\frac{1}{\phi(q)}\sum_{\substack{n\sim N\\ (n,d q)=1}}\alpha_n\Bigr|\ll_A \frac{N\tau(d)^{B_0}}{(\log{N})^A}.
\label{eq:SiegelWalfisz}
\end{equation}
\end{definition}
%
%
%
%
%
%
%
%
We note that $\alpha_n$ certainly satisfies the Siegel-Walfisz condition if $\alpha_n=1$, if $\alpha_n=\mu(n)$ or if $\alpha_n$ is the indicator function of the primes.
%
%
%
%
%
%
%
%
%
%

\section{Main propositions}
%
%
%
%
%
%
%
%
In this section we prove Theorem \ref{thrm:trilinear} assuming four new technical propositions, which we will then establish over the rest of the paper. We do this by applying a sieve decomposition to break the count of primes in arithmetic progressions into counts of integers with particular prime factorizations, which can then be estimated using the relevant proposition. The sieve decomposition is based on ideas based on Harman's sieve (see \cite{Harman}), but we could have used the Heath-Brown identity and some combinatorial lemmas as an alternative.

Define $S_n$ and $S_d(z)$ (depending on integers $a,q$ satisfying $(a,q)=1$ which we suppress for convenience) for integers $n,d$ and a real $z$ by
\begin{align*}
S_n&:=\mathbf{1}_{n\equiv a\Mod{q}}-\frac{1}{\phi(q)}\mathbf{1}_{(n,q)=1},\\
S_d(z)&:=\sum_{\substack{n\sim x/d\\ P^-(n)>z}}S_{d n}
\end{align*}
where the modulus is understood $q=q_1q_2$ or $q=q_0q_1q_2$, in context.
With this notation, we may now state our main propositions.
%
%
%
%
%
%
%
%

The first result is a variant of \cite[Proposition 7.1]{JM1} for trilinear forms of moduli. We prove this in Section \ref{sec:TypeIII2}.

\begin{proposition}[Type II estimate]\label{prpstn:TypeIII2}
Let $A>0$ and let $Q_1,Q_2,Q_3$ satisfy
\begin{align}\label{eq:PropQ5.1}
Q_1Q_2 < x^{1/2+\eps},\qquad
Q_1Q_3 < x^{1/2-2\eps},\qquad
Q_3 < Q_2 < x^{1/32-\eps}.
\end{align}
Let $P_1,\dots,P_J\ge x^{1/7+10\eps}$ be such that $P_1\cdots P_J\asymp x$ and
\[
x^{3/7+\eps}\le\prod_{j\in\mathcal{J}}P_j\le x^{4/7-\eps}
\]
for some subset $\mathcal{J}\subseteq\{1,\dots,J\}$. 
Let $\lambda_q,\nu_q,\eta_q$ be complex sequences with $|\lambda_{q}|,|\nu_q|,|\eta_{q}|\le \tau(q)^{B_0}$.

Then we have 
\[
\underset{(q_1q_2q_3,a)=1}{\sum_{q_1\sim Q_1}\sum_{q_2\sim Q_2}\sum_{q_3\sim Q_3}}
\lambda_{q_1}\nu_{q_2}\eta_{q_3}
\mathop{\sideset{}{^*}\sum}_{\substack{p_1,\dots,p_J\\ p_i\sim P_i\,\forall i}}S_{p_1\cdots p_J}(p_J) \ll_A\frac{x}{(\log{x})^A}.
\]
Here $\sum^*$ indicates that the summation is restricted by $O(1)$ inequalities of the form $p_1^{\alpha_1}\cdots p_J^{\alpha_J}\le B$. The implied constant may depend on all such exponents $\alpha_i$, but none of the quantities $B$.
\end{proposition}
%
%
%
%
%
%
%
%

The second result is a variant of \cite[Proposition 7.2]{JM1} for quadrilinear forms of moduli. We prove this in Section \ref{sec:SieveAsymptotic}.

\begin{proposition}[Sieve asymptotics]\label{prpstn:SieveAsymptotic}
Let $A>0$. Let $x^{3/7+\eps}\ge P_1\ge \dots \ge P_r\ge x^{1/7+10\eps}$ be such that $P_1\cdots P_r\le x^{2/3}$ and such that either $r=1$ or $P_r\le x^{1/4+\eps}$. Let $Q_1,Q_2,Q_3$ satisfy
\begin{align}\label{eq:PropQ5.2}
Q_1Q_2 < x^{1/2+\eps},\qquad  
Q_1Q_3^2 < x^{1/2-2\eps},\qquad
Q_3^2 < Q_2 < x^{1/32-\eps}.
\end{align}
Let $\eta_q,\lambda_q,\nu_q,\mu_q$ be complex sequences with $|\eta_{q}|,|\lambda_{q}|,|\nu_q|,|\mu_q|\le \tau(q)^{B_0}$.

Then we have
\[
\underset{(q_1q_2q_3q_4,a)=1}{\sum_{q_1\sim Q_1}\sum_{q_2\sim Q_2}\sum_{q_3,q_4\sim Q_3}}
\lambda_{q_1}\nu_{q_2}\eta_{q_3}\mu_{q_4} \mathop{\sideset{}{^*}\sum}_{\substack{p_1,\dots,p_r\\ p_i\sim P_i\,\forall i}} S_{p_1\cdots p_r}(x^{1/7+10\eps})\ll_{A}\frac{x}{(\log{x})^A}.
\]
Here $\sum^*$ means that the summation is restricted to $O(1)$ inequalities of the form $p_1^{\alpha_1}\cdots p_r^{\alpha_r}\le B$ for some constants $\alpha_1,\dots \alpha_r$. The implied constant may depend on all such exponents $\alpha_i$, but none of the quantities $B$.

Moreover, we also have the related estimate
\[
\underset{(q_1q_2q_3q_4,a)=1}{\sum_{q_1\sim Q_1}\sum_{q_2\sim Q_2}\sum_{q_3,q_4\sim Q_3}}
\lambda_{q_1}\nu_{q_2}\eta_{q_3}\mu_{q_4} S_{1}(x^{1/7+10\eps})\ll_{A}\frac{x}{(\log{x})^A}.
\]
\end{proposition}
%
%
%
%
%
%
%
%
The third result is a variant of \cite[Proposition 7.3]{JM1} for trilinear forms of moduli. We prove this in Section \ref{sec:4Primes}.

\begin{proposition}[Numbers with 4 or more prime factors]\label{prpstn:4Primes}
Let $A>0$. Let $J\ge 4$ and $P_1\ge\dots\ge P_J\ge x^{1/7+10\eps}$ with $P_1\cdots P_J\asymp x$. Let $Q_1,Q_2,Q_3$ satisfy
\begin{align*}\label{eq:PropQ5.3}
Q_1Q_2 < x^{1/2+\eps},\qquad  
Q_1Q_3 < x^{1/2-2\eps},\qquad
Q_3 < Q_2 < x^{1/32-\eps}.
\end{align*}
Let $\eta_q,\lambda_q,\nu_q$ be complex sequences with $|\eta_{q}|,|\lambda_{q}|,|\nu_q|\le \tau(q)^{B_0}$.

Then we have
\[
\underset{(q_1q_2q_3,a)=1}{\sum_{q_1\sim Q_1}\sum_{q_2\sim Q_2}\sum_{q_3\sim Q_3}}
\lambda_{q_1}\nu_{q_2}\eta_{q_3}
\mathop{\sideset{}{^*}\sum}_{\substack{p_1,\dots,p_J\\ p_i\sim P_i\forall i}}S_{p_1\cdots p_J} \ll_{A}\frac{x}{(\log{x})^A}.
\]
Here $\sum^*$ indicates that the summation is restricted by $O(1)$ inequalities of the form $p_1^{\alpha_1} \cdots p_J^{\alpha_J}\le B$. The implied constant may depend on all such exponents $\alpha_i$, but none of the quantities $B$.
\end{proposition}
%
%
%
%
%
%
%
%
The final result is a variant of \cite[Proposition 7.4]{JM1}. We prove this in Section \ref{sec:3Primes}.

\begin{proposition}[Numbers with three prime factors]\label{prpstn:3Primes}
Let $A>0$ and let $P_1,P_2,P_3\in [x^{1/4},x^{3/7+\eps}]$ with $P_1P_2P_3\asymp x$. Let $Q_1,Q_2,Q_3$ satisfy
\begin{align}\label{eq:PropQ5.4}
Q_1Q_2 < x^{1/2+\eps},\qquad  
Q_1Q_3^2 < x^{1/2-2\eps},\qquad
Q_3^2 < Q_2 < x^{1/32-\eps}.
\end{align}
Let $\eta_q,\lambda_q,\nu_q,\mu_q$ be complex sequences with $|\eta_{q}|,|\lambda_{q}|,|\nu_q|,|\mu_q|\le \tau(q)^{B_0}$. Then we have
\[
\underset{(q_1q_2q_3q_4,a)=1}{\sum_{q_1\sim Q_1}\sum_{q_2\sim Q_2}\sum_{q_3,q_4\sim Q_3}}
\lambda_{q_1}\nu_{q_2}\eta_{q_3}\mu_{q_4}
\mathop{\sideset{}{^*}\sum}_{\substack{p_1,p_2,p_3\\ p_i\sim P_i\forall i}}S_{p_1p_2p_3}\ll_A \frac{x}{(\log{x})^A}.
\]
Here $\sum^*$ means that the summation is restricted to $O(1)$ inequalities of the form $p_1^{\alpha_1}p_2^{\alpha_2} p_3^{\alpha_3}\le B$ for some constants $\alpha_1,\alpha_2, \alpha_3$. The implied constant may depend on all such exponents $\alpha_i$, but none of the quantities $B$.
\end{proposition}
%
%
%
%
%
%
%
%
\begin{proof}[Proof of Theorem \ref{thrm:trilinear} assuming Propositions \ref{prpstn:TypeIII2}, \ref{prpstn:SieveAsymptotic}, \ref{prpstn:4Primes} and \ref{prpstn:3Primes}]
This follows just as in the proof of \cite[Theorem 1.1]{JM1}, except for $(Q_1,Q_2)$ replaced by $(Q_2Q_3Q_4,\, Q_1)$, and with quadrilinear weights $\lambda_{q_1}\nu_{q_2}\eta_{q_3}\mu_{q_4}$ instead of absolute values. In this case, the
Type II estimate and sieve asymptotics in Propositions \ref{prpstn:TypeIII2}, \ref{prpstn:SieveAsymptotic}, \ref{prpstn:3Primes} replace that of \cite[Propositions 7.1, 7.2, 7.4]{JM1}, respectively.
\end{proof}
%
%
%
%
%
%
%
\section{Preliminary Lemmas}
%
%
%
%
%
%
%
In this section, we collect statements of some preliminary lemmas, which will be of use moving forward.
%
%
%
%
%
%
%
\begin{lemma}[Divisor function bounds]\label{lmm:Divisor}
Let $|b|< x-y$ and $y\ge q x^\eps$. Then we have
\[
\sum_{\substack{x-y\le n\le x\\ n\equiv a\Mod{q}}}\tau(n)^C\tau(n-b)^C\ll \frac{y}{q} (\tau(q)\log{x})^{O_{C}(1)}.
\]
\end{lemma}
\begin{proof}
This follows from Shiu's Theorem \cite{Shiu}, and is given in \cite[Lemma 12]{BFI2}.
\end{proof}
%
%
%
%
%
%
%
%
\begin{lemma}[Small sets contribute negligibly]\label{lmm:SmallSets}
Let $\delta>0$, $Q\le x^{1-\eps}$ and let $\mathcal{A}\subseteq [x,2x]$. Then we have
\[
\sum_{q\sim Q}\tau(q)\Bigl|\sum_{\substack{n\in \mathcal{A}\\ n\equiv a\Mod{q}}}1-\frac{1}{\phi(q)}\sum_{\substack{n\in\mathcal{A}\\ (n,q)=1}}1\Bigr|\ll x^\delta\#\mathcal{A}^{1-\delta}(\log{x})^{O_\delta(1)}.
\]
\end{lemma}
\begin{proof}
See \cite[Lemma 8.9]{JM1}.
\end{proof}
%
%
%
%
%
%
%
%
\begin{lemma}[Separation of variables from inequalities]\label{lmm:Separation}
Let $Q_1Q_2\le x^{1-\eps}$. Let $N_1,\dots, N_r\ge z_0$ satisfy $N_1\cdots N_r\asymp x$. Let $\alpha_{n_1,\dots,n_r}$ be a complex sequence with $|\alpha_{n_1,\dots,n_r}|\le (\tau(n_1)\cdots \tau(n_r))^{B_0}$. Then, for any choice of $A>0$ there is a constant $C=C(A,B_0,r)$ and intervals $\mathcal{I}_1,\dots,\mathcal{I}_r$ with $\mathcal{I}_j\subseteq [P_j,2P_j]$ of length $\le P_j(\log{x})^{-C}$ such that
\begin{align*}
\sum_{q_1\sim Q_1}\sum_{\substack{q_2\sim Q_2\\ (q_1q_2,a)=1}}&\Bigl|\mathop{\sideset{}{^*}\sum}_{\substack{n_1,\dots,n_r\\ n_i\sim N_i\forall i}}\alpha_{n_1,\dots,n_r}S_{n_1\cdots n_r}\Bigr|\\
&\ll_r \frac{x}{(\log{x})^A}+(\log{x})^{r C}\sum_{q_1\sim Q_1}\sum_{\substack{q_2\sim Q_2\\ (q_1q_2,a)=1}}\Bigl|\sum_{\substack{n_1,\dots,n_r\\ n_i\in \mathcal{I}_i\forall i}}\alpha_{n_1,\dots,n_r}S_{n_1\cdots n_r}\Bigr|.
\end{align*}
Here $\sum^*$ means that the summation is restricted to $O(1)$ inequalities of the form $n_1^{\alpha_1}\cdots n_r^{\alpha_r}\le B$ for some constants $\alpha_1,\dots \alpha_r$ and some quantity $B$.  The implied constant may depend on all such exponents $\alpha_i$, but none of the quantities $B$.
\end{lemma}
\begin{proof}
See \cite[Lemma 8.10]{JM1}.
\end{proof}

%
%
%
%
%
%
%
\begin{lemma}\label{lmm:SiegelWalfiszMaintain}
Let $C,B>0$ be constants and let $\alpha_n$ be a sequence satisfing the Siegel-Walfisz condition \eqref{eq:SiegelWalfisz}, supported on $n\le 2x$ with $P^-(n)\ge z_0=x^{1/(\log\log{x})^3}$ and satisfying $|\alpha_n|\le \tau(n)^B$. Then $\mathbf{1}_{\tau(n)\le (\log{x})^C}\alpha_n$ also satisfies the Siegel-Walfisz condition.
\end{lemma}
\begin{proof}
See \cite[Lemma 13.7]{JM1}.
\end{proof}
%
%
%
%
%
%
%
%
\begin{lemma}[Most moduli have small $z_0$-smooth part]\label{lmm:RoughModuli}
Let $Q<x^{1-\eps}$. Let $\gamma_b,c_q$ be complex sequences with $|\gamma_b|,|c_b|\le \tau(n)^{B_0}$ and recall $z_0:=x^{1/(\log\log{x})^3}$ and $y_0:=x^{1/\log\log{x}}$. Let $\operatorname{sm}(n;z)$ denote the $z$-smooth part of $n$. (i.e. $\operatorname{sm}(n;z)=\prod_{p\le z}p^{\nu_p(n)}$). Then for every $A>0$ we have that
\[
\sum_{\substack{q\sim Q\\ \operatorname{sm}(q;z_0)\ge y_0}}c_q\sum_{b\le  x}\gamma_b\Bigl(\mathbf{1}_{b\equiv a\Mod{q}}-\frac{\mathbf{1}_{(b,q)=1}}{\phi(q)}\Bigr)\ll_{A,B_0} \frac{x}{(\log{x})^A}.
\]
\end{lemma}
\begin{proof}
See \cite[Lemma 13.10]{JM1}.
\end{proof}
%
%
%
%
%
%
%
%
\begin{lemma}[Splitting into coprime sets]\label{lmm:FouvryDecomposition}
Let $\mathcal{N}\subseteq \mathbb{Z}_{>0}^2$ be a set of pairs $(a,b)$ satisfying:
\begin{enumerate}
\item $a,b\le x^{O(1)}$,
\item $\gcd(a,b)=1$,
\item The number of prime factors of $a$ and of $b$ is $\ll (\log\log{x})^3$. 
\end{enumerate}
Then there is a partition $\mathcal{N}=\mathcal{N}_1\sqcup\mathcal{N}_2\sqcup\dots \sqcup\mathcal{N}_J$ into $J$ disjoint subsets with
\[
J\ll \exp\Bigl(O(\log\log{x})^4\Bigr),
\]
such that if $(a,b)$ and $(a',b')$ are in the same set $\mathcal{N}_j$, then $\gcd(a,b')=\gcd(a',b)=1$. 
\end{lemma}
\begin{proof}
This follows immediately from \cite[Lemme 6]{Fouvry}. Also see \cite[Lemma 13.2]{JM1}
\end{proof}
%
%
%
%
%
%
%
%

\section{Exponential sum estimates}
%
%
%
%
%
%
%
%
In this section, we cite important estimates for several exponential sums.

\begin{lemma}[Weil bound for Kloosterman sums]\label{lmm:Kloosterman}
Let $S(m,n;c)$ be the Kloosterman sum 
\[
S(m,n;c):=\sum_{\substack{b\Mod{c}\\ (b,c)=1}}e\Bigl(\frac{m b+ n \overline{b}}{c}\Bigr).
\]
Then we have that
\[
S(m,n;c)\ll \tau(c)c^{1/2}\gcd(m,n,c)^{1/2}.
\]
\end{lemma}
\begin{proof}
This is \cite[Corollary 11.12]{IwaniecKowalski}. Also see \cite[Lemma 13.3]{JM1}
\end{proof}
%
%
%
%
%
%
%
%
\begin{lemma}[Completion of inverses]\label{lmm:InverseCompletion}
Let $C>0$ and $f:\mathbb{R}\rightarrow\mathbb{R}$ be a smooth function which is supported on $[-10,10]$ and satisfies $\|f^{(j)}\|_\infty\ll_j (\log{x})^{j C}$ for all $j\ge 0$. Let $(d,q)=1$. Then we have for any $H\ge x^\eps d q/N$
\begin{align*}
\sum_{\substack{(n,q)=1\\ n\equiv n_0\Mod{d}}}&f\Bigl(\frac{n}{N}\Bigr)e\Bigl(\frac{b\overline{n}}{q}\Bigr)=\frac{N\hat{f}(0)}{d q}\sum_{(c,q)=1}e\Bigl(\frac{b c}{q}\Bigr)\\
&+\frac{N}{d q}\sum_{1\le |h|\le H}\hat{f}\Bigl(\frac{h N}{d q}\Bigr)e\Bigl(\frac{n_0\overline{q}h}{d}\Bigr)\sum_{\substack{c\Mod{q}\\ (c,q)=1}} e\Bigl(\frac{b\overline{ dc}+h c}{q}\Bigr)+O_C(x^{-100}).
\end{align*}
Moreover, if $\|f^{(j)}\|_\infty\ll ((j+1)\log{x})^{jC}$ then we have the same result for any $H\ge (\log{x})^{2C+1} dq/N$.
\end{lemma}
\begin{proof}
See \cite[Lemma 13.5]{JM1}
\end{proof}

%
%
%
%
%
%
%
%
\begin{proposition}[Reduction to exponential sums]\label{prpstn:GeneralDispersion}
Let $\alpha_n,\beta_m,\gamma_{q,d},\lambda_{q,d,r}$ be complex sequences with $|\alpha_n|,|\beta_n|\le \tau(n)^{B_0}$ and $|\gamma_{q,d}|\le \tau(q d)^{B_0}$ and $|\lambda_{q,d,r}|\le \tau(q d r)^{B_0}$. Let $\alpha_n$ and $\lambda_{q,d,r}$ be supported on integers with $P^-(n)\ge z_0$ and $P^-(r)\ge z_0$, and let $\alpha_n$ satisfy the Siegel-Walfisz condition \eqref{eq:SiegelWalfisz}. Let
\[
\mathcal{S}:=\sum_{\substack{d\sim D\\ (d,a)=1}}\sum_{\substack{q\sim Q\\ (q,a)=1}}\sum_{\substack{r\sim R\\ (r,a)=1}}\lambda_{q,d,r}\gamma_{q,d}\sum_{m\sim M}\beta_m\sum_{n\sim N}\alpha_n\Bigl(\mathbf{1}_{m n\equiv a\Mod{q r d}}-\frac{\mathbf{1}_{(m n,q r d)=1}}{\phi(q r d)}\Bigr).
\]
Let $A>0$, $1\le E\le x$ and $C=C(A,B_0)$ be sufficiently large in terms of $A,B_0$, and let $N,M$ satisfy
\[
N>Q D E(\log{x})^{C},\qquad M>(\log{x})^C.
\]
Then we have
\[
|\mathcal{S}|\ll_{A,B_0} \frac{x}{(\log{x})^A}+M D^{1/2}Q^{1/2}(\log{x})^{O_{B_0}(1)}\Bigl(|\mathcal{E}_1|^{1/2}+|\mathcal{E}_2|^{1/2}\Bigr),
\]
where
\begin{align*}
\mathcal{E}_{1}&:=\sum_{e\sim E}\mu^2(e)\sum_{\substack{q\\ (q,a)=1}}\sum_{\substack{d\sim D\\ (d,a)=1}}\sum_{\substack{r_1,r_2\sim R\\ (r_1r_2,a)=1}}\psi_0\Bigl(\frac{q}{Q}\Bigr)\frac{\lambda_{q,d,r_1}\overline{\lambda_{q,d,r_2}} }{\phi(q d e r_2)q d r_1}\\
&\qquad \times\sum_{\substack{n_1,n_2\sim N\\ (n_1,q d e r_1)=1\\(n_2,q d e  r_2)=1}}\alpha_{n_1}\overline{\alpha_{n_2}}\sum_{1\le |h|\le H_1}\hat{\psi}_0\Bigl(\frac{h M}{q d r_1}\Bigr)e\Bigl( \frac{a h \overline{ n_1}}{q d r_1}\Bigr),\\
\mathcal{E}_2&:=\sum_{e\sim E}\mu^2(e)\sum_{\substack{q\\ (q,a)=1}}\psi_0\Bigl(\frac{q}{Q}\Bigr)\sum_{\substack{d\sim D\\ (d,a)=1}}\sum_{\substack{r_1,r_2\sim R\\ (r_1,a r_2)=1\\ (r_2,a q d r_1)=1}}\frac{\lambda_{q,d,r_1}\overline{\lambda_{q,d,r_2}}}{q d r_1 r_2}\\
&\qquad \times\sum_{\substack{n_1,n_2\sim N\\ n_1\equiv n_2\Mod{q d e}\\ (n_1,n_2 e q d r_1)=1\\(n_2,n_1 e q d r_2)=1\\ |n_1-n_2|\ge N/(\log{x})^C}}\alpha_{n_1}\overline{\alpha_{n_2}}\sum_{1\le |h|\le H_2}\hat{\psi}_0\Bigl(\frac{h M}{q d r_1 r_2}\Bigr)e\Bigl(\frac{ah\overline{n_1r_2}}{q d r_1}+\frac{ah\overline{n_2 q d r_1}}{r_2}\Bigr),\\
H_1&:=\frac{Q D R}{M}\log^5{x},\\
H_2&:=\frac{Q D R^2}{M}\log^5{x}.
\end{align*}
\end{proposition}
\begin{proof}
This is \cite[Theorem 14.4]{JM1}.
\end{proof}
%
%
%
%
%
%
%
%

\begin{lemma}[Simplification of exponential sum]\label{lmm:Simplification}
Let $N,M,Q,R,S \le x$ with $NM\asymp x$ and 
\begin{align}
Q R&<x^{2/3},\label{eq:CrudeSize}\\
Q R^2 &< M x^{1-2\eps}.\label{eq:CrudeSize2}
\end{align}

Let $\lambda_{q,r}$ and $\alpha_n$ be complex sequences supported on $P^-(n),P^-(r)\ge z_0$ with $|\lambda_{q,r}|\le \tau(qr)^{B_0}$ and $|\alpha_n|\le \tau(n)^{B_0}$. Let $H:=\frac{Q R^2}{M}\log^5{x}$ and let
\begin{align*}
\mathcal{E}&:=\sum_{\substack{ (q,a)=1}}\psi_0\Bigl(\frac{q}{Q}\Bigr)\sum_{\substack{r_1,r_2\sim R\\ (r_1,a r_2)=1\\ (r_2,a q r_2)=1}}\frac{\lambda_{q,r_1}\overline{\lambda_{q,r_2}}}{q r_1 r_2}\sum_{\substack{n_1,n_2\sim N\\ n_1\equiv n_2\Mod{q}\\ (n_1,n_2qr_1)=1\\(n_2,n_1qr_2)=1\\ |n_1-n_2|\ge N/(\log{x})^C}}\alpha_{n_1}\overline{\alpha_{n_2}}\\
&\qquad\qquad \times\sum_{1\le |h|\le H}\hat{\psi}_0\Bigl(\frac{h M}{q r_1 r_2}\Bigr)e\Bigl(\frac{ah\overline{n_1 r_2}}{q r_1}+\frac{ah\overline{n_2 q r_1}}{r_2}\Bigr).
\end{align*}
Then we have (uniformly in $C$)
\[
\mathcal{E}\ll_{B_0}\exp((\log\log{x})^5)\sup_{\substack{H'\le H\\ Q'\le 2Q\\ R_1,R_2\le 2R}}|\mathcal{E}'|+\frac{N^2}{Qx^\eps},
\]
where
\[
\mathcal{E}'=\sum_{\substack{Q\le q\le Q'\\ (q,a)=1}}\sum_{\substack{R\le r_1\le  R_1\\ R\le r_2\le R_2\\ (r_1a r_2)=1\\ (r_2,a q r_1)=1}}\frac{\lambda_{q,r_1}\overline{\lambda_{q,r_2}}}{q r_1 r_2}\sum_{\substack{n_1,n_2\sim N\\ n_1\equiv n_2\Mod{q}\\ (n_1,qr_1n_2)=1\\ (n_2,qr_2n_1)=1\\ (n_1r_2,n_2)\in\mathcal{N}\\ |n_1-n_2|\ge N/(\log{x})^C}}\alpha_{n_1}\overline{\alpha_{n_2}}\sum_{1\le |h| \le H'} e\Bigl(\frac{ ah\overline{n_2 q r_1}(n_1-n_2)}{n_1 r_2}\Bigr),
\]
and $\mathcal{N}$ is a set with the property that if $(a,b)\in\mathcal{N}$ and $(a',b')\in\mathcal{N}$ then we have $\gcd(a,b')=\gcd(a',b)=1$.
\end{lemma}
\begin{proof}
This is \cite[Theorem 14.5]{JM1}.
\end{proof}

\begin{lemma}[Deshouillers--Iwaniec]\label{lmm:DeshouillersIwaniec}
Let $b_{n,r,s}$ be a 1-bounded sequence and $R,S,N,D,C\ll x^{O(1)}$. Let $g(c,d)=g_0(c/C,d/D)$ where $g_0$ is a smooth function supported on $[1/2,5/2]\times [1/2,5/2]$. Then we have
\[
\sum_{r\sim R} \sum_{\substack{s\sim S\\ (r,s)=1}}\sum_{n\sim N}b_{n,r,s}\sum_{d\sim D}\sum_{\substack{c\sim C\\ (rd,sc)=1}}g(c,d) e\Bigl(\frac{n\overline{dr}}{cs}\Bigr)\ll_{g_0} x^\eps \Bigl(\sum_{r\sim R}\sum_{s\sim S}\sum_{n\sim N}|b_{n,r,s}|^2\Bigr)^{1/2}\mathcal{J}.
\]
where
\[
\mathcal{J}^2=CS(RS+N)(C+DR)+C^2 D S\sqrt{(RS+N)R}+D^2NR.
\]
\end{lemma}
\begin{proof}
This is \cite[Theorem 15.1]{JM1}.
Also see \cite[Theorem 12]{DeshouillersIwaniec} (correcting a minor typo in the last term of $\mathcal{J}^2$ which had erroneously written $D^2NR/S$).
\end{proof}

%
%
%
%
%
%
%
%
\section{Zhang-style estimates}
%
%
%
%
%
%
%

In this section we establish a new Zhang-style exponential sum estimate.

%
%
%
%
%
%
%
%

\begin{lemma}[Zhang exponential sum estimate]\label{lmm:Zhang1}
Let $Q,R,S,M,N$ satisfy $NM\asymp x$ and
\begin{align*}
Q^{7} R^{12}S^{10}< x^{4-18\eps},\qquad 
Q <N<\frac{x^{1-5\eps}}{QS^2},
\end{align*}
and let  $H\ll QNR^2S^2/x^{1-\eps}$. Let $\lambda_{q,s}$, $\nu_r$ and $\alpha_n$ be $1$-bounded complex sequences supported on $r,s$ with $P^-(rs)\ge z_0$. Let
\begin{align*}
\mathcal{Z}&:=\sum_{\substack{q\sim Q\\ (q,a)=1}}\sum_{s_1,s_2\sim S}\sum_{\substack{r_1,r_2\sim R\\ (r_1s_1,ar_2s_2)=1\\ (r_2s_2,aqr_1s_1)=1}}\sum_{\substack{n_1,n_2\sim N\\ n_1\equiv n_2\Mod{q}\\ (n_1,q r_ 1s_1 n_2)=1\\ (n_2,q r_2s_2 n_1)=1\\ |n_1-n_2|\ge N/(\log{x})^C}}\frac{\alpha_{n_1}\overline{\alpha_{n_2}} \lambda_{q,s_1}\nu_{r_1}\overline{\lambda_{q,s_2}\nu_{r_2}}}{q r_1 r_2s_1s_2}\\
&\qquad\times \sum_{1\le |h|\le H}\hat{\psi_0}\Bigl(\frac{h M}{q r_1 r_2s_1s_2}\Bigr)e\Bigl(ah\Bigl(\frac{\overline{n_1 r_2s_2}}{qr_1s_1}+\frac{\overline{n_2 q r_1s_1}}{r_2s_2}\Bigr)\Bigr),
\end{align*}
Then we have
\[
\mathcal{Z}\ll \frac{N^2 }{Q x^\eps}.
\]
\end{lemma}
\begin{proof}
We assume throughout that $H\ll QR^2S^2N/x^{1-\eps}$ and that $Q\ll N$, and deduce the other conditions are sufficient to give the result.

Since we only consider $r_1,r_2,s_1,s_2$ with $P^-(r_1r_2s_1s_2)\ge z_0$, we see $r_1,r_2,s_1,s_2$ have at most $(\log\log{x})^3$ prime factors. Therefore, by Lemma \ref{lmm:FouvryDecomposition}, there are $O(\exp(\log\log{x})^5))$ different sets $\mathcal{N}_1,\mathcal{N}_2,\dots$ which cover all possible pairs $(r_1s_1,r_2s_2)$, and such that if $(r_1s_1,r_2s_2)$ and $(r_1's_1,r_2's_2)\in\mathcal{N}_j$ then $\gcd(r_1s_1,r_2's_2)=\gcd(r_1's_1,r_2s_2)=1$. Taking the worst such set $\mathcal{N}$, we see that
 \begin{align*}
\mathcal{Z}&\ll \exp((\log\log{x})^5)\Bigl|\sum_{\substack{q\sim Q\\ (q,a)=1}} \sum_{s_1,s_2\sim S} \sum_{\substack{r_1,r_2\sim R\\ (r_1s_1,a r_2s_2)=1\\ (r_2s_2,a q r_1s_1)=1\\ (r_1s_1,r_2s_2)\in\mathcal{N}}}\sum_{\substack{n_1,n_2\sim N\\ n_1\equiv n_2\Mod{q}\\ (n_1,q r_1s_1 n_2)=1\\ (n_2,q r_2s_2 n_1)=1\\ |n_1-n_2|\ge N/(\log{x})^C}}\frac{\alpha_{n_1}\overline{\alpha_{n_2}} \lambda_{q,s_1}\nu_{r_1}\overline{\lambda_{q,s_2}\nu_{r_2}}}{q r_1 r_2s_1s_2}\\
&\qquad\qquad\times\sum_{1\le |h|\le H}\hat{\psi_0}\Bigl(\frac{h M}{q r_1 r_2s_1s_2}\Bigr)e\Bigl(ah\Bigl(\frac{\overline{n_1 r_2s_2}}{qr_1s_1}+\frac{\overline{n_2 q r_1s_1}}{r_2s_2}\Bigr)\Bigr)\Bigr|.
\end{align*}
We now Cauchy in $n_1,n_2, s_1,s_2,q$ to eliminate the $\alpha_n,\lambda_{q,s}$ coefficients and insert a smooth majorant for the $n_1$ and $n_2$ summations. This gives (using $Q\ll N$)
\[
\mathcal{Z}^2\ll \exp(2(\log\log{x})^5)\Bigl(\sum_{q\sim Q}\sum_{s_1,s_2\sim S} \sum_{\substack{n_1,n_2\sim N\\ n_1\equiv n_2\Mod{q}}}\frac{1}{(qs_1s_2)^2}\Bigr)|\mathcal{Z}_2|\ll \frac{x^\eps N^2}{Q^2S^2}|\mathcal{Z}_2|,
\]
where
\begin{align*}
\mathcal{Z}_2&:=\sum_{\substack{q\sim Q\\ (q,a)=1}}\sum_{s_1,s_2\sim S} \sum_{\substack{n_1,n_2\\ (n_1n_2,q)=1}} \psi_0\Bigl(\frac{n_1}{N}\Bigr)\psi_0\Bigl(\frac{n_2}{N}\Bigr)\\
&\qquad \times\Bigl|\sum_{\substack{r_1,r_2\sim R\\ (r_1s_1,a r_2s_2 n_1)=1\\ (r_2s_2,a q r_1s_1 n_2)=1\\ (r_1s_1,r_2s_2)\in\mathcal{N}}}\frac{\nu_{r_1}\overline{\nu_{r_2}}}{r_1 r_2} \sum_{1\le |h|\le H}\hat{\psi_0}\Bigl(\frac{h M}{q r_1 r_2s_1s_2}\Bigr) e\Bigl(ah\Bigl(\frac{\overline{n_1 r_2s_2}}{q r_1s_1}+\frac{\overline{ n_2q r_1s_1}}{r_2s_2}\Bigr)\Bigr)\Bigr|^2.
\end{align*}
Note $\mathcal{Z}_2 \le |\mathcal{Z}_3|/R^4$ for
\begin{align*}
\mathcal{Z}_3:=\sum_{q\sim Q}\sum_{s_1,s_2\sim S} \sum_{\substack{r_1,r_1',r_2,r_2'\sim R\\ (qr_1r_1',r_2r_2')=1}} \sum_{1\le |h|,|h'|\le H}\Bigl|\sum_{\substack{n_1,n_2\\ n_1\equiv n_2\Mod{q} \\ (n_1,qr_1r_1's_1)=1\\ (n_2,r_2r_2's_2)=1}} \psi_0\Bigl(\frac{n_1}{N}\Bigr)\psi_0\Bigl(\frac{n_2}{N}\Bigr) e\Bigl(\frac{c_1\overline{n_1}}{q r_1r_1's_1}+\frac{c_2\overline{n_2}}{r_2 r_2's_2}\Bigr)\Bigr|,
\end{align*}
and where $c_1\Mod{q r_1r_1's_1}$ and $c_2\Mod{r_2r_2's_2}$ are given by
\begin{align*}
c_1&=a(h r_1'r_2'-h' r_1r_2)\overline{r_2r_2's_2},\\
c_2&=a (h r_1'r_2'-h' r_1r_2)\overline{q r_1r_1's_1}.
\end{align*}
(Here we used the fact that $(r_1s_1,r_2s_2),(r_1's_1,r_2's_2)\in\mathcal{N}$ to conclude $(r_1s_1,r_2's_2)=(r_1's_1,r_2s_2)=1$.) In order to establish the desired bound $\mathcal{Z}\ll N^2/(x^\eps Q)$, it suffices to show $\mathcal{Z}_2\ll N^2S^2/x^{3\eps}$, and so it suffices to prove
\begin{equation}
\mathcal{Z}_3\ll \frac{N^2 S^2 R^4}{x^{3\eps}}.\label{eq:ZhangE4}
\end{equation}
We separate the diagonal terms $\mathcal{Z}_{=}$ with $hr_1'r_2'=h'r_1r_2$ and the off-diagonal terms $\mathcal{Z}_{\ne}$ with $h r_1'r_2'\ne h'r_1r_2$.
\begin{equation}
\mathcal{Z}_3\ll \mathcal{Z}_{=}+\mathcal{Z}_{\ne}.
\label{eq:Z3Bound}
\end{equation}
We first consider the diagonal terms $\mathcal{Z}_{=}$. Given a choice of $h,r_1',r_2'$ there are $x^{o(1)}$ choices of $h',r_1,r_2$ by the divisor bound. Thus, estimating the remaining sums trivially we have (using $Q\ll N$ and $H\ll NQ(RS)^2/x^{1-\eps}$)
\begin{equation}
\mathcal{Z}_{=}\ll x^{o(1)} Q (RS)^2 H N \Bigl(\frac{N}{Q}+1\Bigr)\ll \frac{N^3 Q (RS)^4}{x^{1-2\eps}}.
\label{eq:ZEq}
\end{equation}
Now we consider the off-diagonal terms $\mathcal{Z}_{\ne}$. By Lemma \ref{lmm:InverseCompletion}, we have that
\begin{align*}
&\sum_{\substack{n_2\equiv n_1\Mod{q}\\ (n_2,r_2r_2's_2)=1}}\psi_0\Bigl(\frac{n_2}{N}\Bigr)e\Bigl(\frac{c_2\overline{n_2}}{r_2 r_2's_2}\Bigr)\\
&\qquad=\frac{N}{q r_2 r_2's_2}\sum_{|\ell_2|\le x^\eps QSR^2/N}\hat{\psi_0}\Bigl(\frac{\ell_2 N}{q r_2r_2's_2}\Bigr)S(c_2,\ell_2\overline{q};r_2r_2's_2)e\Bigl(\frac{\ell_2 n_1 \overline{r_2r_2's_2}}{q}\Bigr) +O(x^{-100}).
\end{align*}
here $S(m,n;c)$ is the standard Kloosterman sum. By Lemma \ref{lmm:InverseCompletion} again, we have that
\begin{align*}
&\sum_{(n_1,q r_1 r_1's_1)=1}\psi_0\Bigl(\frac{n_1}{N}\Bigr)e\Bigl(\frac{c_1\overline{n_2}+\ell_2 r_1 r_1's_1 n_1\overline{r_2r_2's_2} }{q r_1 r_1's_1}\Bigr)\\
&=\frac{N}{q r_1r_1's_1}\sum_{|\ell_1|\le x^\eps QSR^2/N} \hat{\psi_0}\Bigl(\frac{\ell_1 N}{q r_1 r_1's_1}\Bigr) S(c_1,\ell_1\overline{q};r_1 r_1's_1) S(c_1,\ell_1\overline{r_1 r_1's_1}+\ell_2\overline{r_2r_2's_2};q)+O(x^{-100}).
\end{align*}
Thus, we see that $\mathcal{Z}_{3}$ is a sum of Kloosterman sums. By the standard Kloosterman sum bound of Lemma \ref{lmm:Kloosterman} $S(m,n;c)\ll \tau(c) c^{1/2}(m,n,c)^{1/2}\ll c^{1/2+o(1)}(m,c)^{1/2}$, the inner sum has the bound
\begin{align*}
\sum_{\substack{n_1,n_2\\ n_1\equiv n_2\Mod{q} \\ (n_1,qr_1r_1's_1)=1\\ (n_2,r_2r_2's_2)=1}}&\psi_0\Bigl(\frac{n_1}{N}\Bigr)\psi_0\Bigl(\frac{n_2}{N}\Bigr)e\Bigl(\frac{c_1\overline{n_1}}{q r_1r_1's_1}+\frac{c_2\overline{n_2}}{r_2 r_2's_2}\Bigr)\\
&\ll \frac{x^{o(1)}N^2}{Q^2 S^2R^4}\sum_{\substack{|\ell_1|\le x^\eps QSR^2/N\\ |\ell_2|\le x^\eps QSR^2/N}}(QR^4S^2)^{1/2}\cdot (c_2,r_2r_2's_2)^{1/2}(c_1,qr_1r_1's_1)^{1/2}\\
&\ll x^{3\eps}Q^{1/2}R^2S\cdot (hr_1'r_2'-h'r_1r_2,qr_1r_1'r_2r_2's_1s_2)^{1/2}\\
\end{align*}
Substituting this into our expression for $\mathcal{Z}_{\ne}$ gives
\begin{align}
\mathcal{Z}_{\ne}
&\ll x^{3\eps} Q^{1/2} R^2S\sum_{s_1,s_2\sim S} \sum_{r_1,r_1'\sim R}\sum_{r_2,r_2'\sim R}\sum_{\substack{1\le |h|,|h'|\le H\\ hr_1'r_2'\ne h'r_1r_2}}  \sum_{q\sim Q}(hr_1'r_2'-h'r_1r_2,qr_1r_1'r_2r_2's_1s_2)^{1/2}\nonumber\\
&\ll x^{3\eps} Q^{1/2} R^2S\,(x^\eps S^2R^4H^2Q) \ = \ x^{4\eps} Q^{3/2} R^6S^3H^2 \nonumber\\
&\ll \frac{N^2 Q^{7/2} R^{10}S^7}{x^{2-6\eps}}.\label{eq:ZNeq}
\end{align}
Substituting \eqref{eq:ZEq} and \eqref{eq:ZNeq} into \eqref{eq:Z3Bound} then gives
\[
\mathcal{Z}_3\ll \frac{N^3 Q R^4S^4}{x^{1-2\eps}}+\frac{N^2 Q^{7/2} R^{10}S^7}{x^{2-6\eps}}.
\]
This gives the desired bound \eqref{eq:ZhangE4} provided we have
\begin{align}
N&<\frac{x^{1-5\eps} }{QS^2},\\
Q^{7} R^{12}S^{10} &<x^{4-18\eps}.
\end{align}
This gives the result.
\end{proof}

%
%
%
%
%
%
%
%
%
%
%
%
%
%
%
%
\begin{lemma}[Second exponential sum estimate]\label{lmm:Zhang2}
Let $Q,R,M,N\le x^{O(1)}$ satisfy $NM\asymp x$ and
\begin{align*}
N Q < x^{1-4\eps},\qquad
N Q^{5/2}R^3< x^{2-4\eps}, \qquad
N^2 Q R &<x^{2-4\eps}.
\end{align*}
Let $\alpha_n$, $\lambda_{q,r}$ be $1$-bounded complex sequences, $H_1=(QR\log^5{x})/M$ and 
\begin{align*}
\widetilde{\mathcal{Z}^{}}&:=\sum_{\substack{q\sim Q \\ (q,a)=1}}\sum_{\substack{r_1,r_2\sim R\\ (r_1 r_2,a)=1}}\frac{\lambda_{q,r_1}\overline{\lambda_{q,r_2}}}{\phi(q r_2)q r_1}\sum_{\substack{n_1,n_2\sim N\\ (n_1,q r_1)=1\\(n_2,q r_2)=1}}\alpha_{n_1}\overline{\alpha_{n_2}}\sum_{1\le |h|\le H_1}\hat{\psi}_0\Bigl(\frac{h M}{q r_1}\Bigr)e\Bigl( \frac{a h \overline{ n_1}}{q r_1}\Bigr).
\end{align*}
Then we have
\[
\widetilde{\mathcal{Z}^{}}\ll \frac{N^2}{Q x^\eps}.
\]
\end{lemma}
\begin{proof}
See \cite[Lemma 17.2]{JM1}
\end{proof}
%
%
%
%
%
%
%
%


We are now able to establish the following Zhang-style estimate. This is a variant of \cite[Proposition 8.2]{JM1}.

\begin{proposition}[Zhang-style estimate]\label{prpstn:Zhang}
Let $A>0$. Let $N,M,Q_1,Q_2,Q_3\ge 1$ with $NM\asymp x$ be such that
\begin{align*}
Q_1^{7}Q_2^{12}Q_3^{10}&<x^{4-20\eps},\qquad
Q_2 < Q_1Q_3^3,\qquad
x^\eps Q_1<N<\frac{x^{1-6\eps}}{Q_1Q_3^2}.
\end{align*}
Let $\beta_m,\alpha_n$ be complex sequences such that $|\alpha_n|,|\beta_n|\le \tau(n)^{B_0}$ and such that $\alpha_{n}$ satisfies the Siegel-Walfisz condition \eqref{eq:SiegelWalfisz} and $\alpha_n$ is supported on $n$ with all prime factors bigger than $z_0=x^{1/(\log\log{x})^3}$. Let $\lambda_{q},\nu_q,\eta_q$ be $1$-bounded complex sequences
\[
\Delta(q):=\sum_{m\sim M}\sum_{\substack{n\sim N}}\alpha_n\beta_m\Bigl(\mathbf{1}_{m n\equiv a \Mod{q}}-\frac{\mathbf{1}_{(m n,q)=1}}{\phi(q)}\Bigr).
\]
Then we have
\[
\mathop{\sum_{q_1\sim Q_1}\sum_{q_2\sim Q_2}\sum_{q_3\sim Q_3}} \limits_{(q_1q_2q_3,a)=1}\lambda_{q_1}\nu_{q_2}\eta_{q_3}\Delta(q_1q_2q_3)\ll_{A,B_0} \frac{x}{(\log{x})^A}.
\]
\end{proposition}
\begin{proof}
First we note that by Lemma \ref{lmm:Divisor} the set of $n,m$ with $\max(|\alpha_n|,|\beta_m|)\ge(\log{x})^C$ has size $\ll x(\log{x})^{O_{B_0}(1)-C}$, so by Lemma \ref{lmm:SmallSets} these terms contribute negligibly if $C=C(A,B_0)$ is large enough. Thus, by dividing through by $(\log{x})^{2C}$ and considering $A+2C$ in place of $A$, it suffices to show the result when all the sequences are 1-bounded. ($\alpha_n$ still satisfies \eqref{eq:SiegelWalfisz} by Lemma \ref{lmm:SiegelWalfiszMaintain}.)

We factor $q_1 = d_1q$, $q_2 = d_2r$, $q_3=d_3s$ into parts with large and small prime factors. By putting these in dyadic intervals, we see that it suffices to show for every $A > 0$ and every choice of $D_1Q \asymp Q_1$, $D_2R \asymp Q_2$, $D_3S \asymp Q_3$ that
\begin{align*}
&\underset{\substack{P^-(qrs) > z_0\\(qrs,a)=1}}{\sum_{q\sim Q}\sum_{r\sim R}\sum_{s\sim S}}
\underset{z_0\ge P^+(d_1d_2d_3)}{\sum_{d_1\sim D_1}\sum_{d_2\sim D_2}\sum_{d_3\sim D_3}} \lambda_{qd_1}\nu_{rd_2}\eta_{sd_3} \Delta(qrsd_1d_2d_3)\ll_{A}\frac{x}{(\log{x})^A}.
\end{align*}
By Lemma \ref{lmm:RoughModuli} we have the result unless $D_1,D_2,D_3\le y_0=x^{1/\log\log{x}}$, may assume that $Q = Q_1x^{-o(1)}$, $R = Q_2x^{-o(1)}$, $S=Q_3x^{-o(1)}$, . We let $d = d_1d_2d_3$, and extend the summation over $d_1,d_2,d_3$ to only have the constraint $d \le y_0^3$ and then insert some divisor-bounded coefficients $c_d$ to absorb the conditions $z_0 \ge P^+(d), d \sim D$. Also we modify the coefficients $\lambda'_q=\lambda_q\1_{P^-(q)>z_0}$, and similarly for $\nu'_r,\eta'_s$.
Thus it suffices to show that
\[
\sum_{\substack{d\le y_0^3\\ (d,a)=1}}\underset{(qrs,a)=1}{\sum_{s\sim S}\sum_{r\sim R}\sum_{q\sim Q}} 
\eta'_{s}\nu'_{r}\lambda'_{q}
c_d\Delta(qrsd)\ll_A\frac{x}{(\log{x})^A}.
\]
If we let
\begin{align*}
\lambda_{b_1,b_2,b_3} = \1_{b_1=1}\sum_{qrs=b_3}\lambda'_{q}\nu'_{r}\eta'_{s}
c_{b_2}
\end{align*}
then we see that we have a sum of the type considered in Proposition \ref{prpstn:GeneralDispersion} (taking `$R$' to be $QR$, `$Q$' to be $D$ and `$E$' to be 1). By the assumptions of the proposition, we have that $NQRS \le NQ_1Q_2Q_3 < x^{-o(1)}$, so we have $H_1 = (QDRS\log^5 x)/M < 1$ and so the sum $\mathcal E_1$ of Proposition \ref{prpstn:GeneralDispersion} vanishes. Therefore, by Proposition \ref{prpstn:GeneralDispersion}, it suffices to show that
\[
\mathcal{E}_2\ll \frac{N^2}{Q x^\eps},
\]
where $H_2=(QR^2S^2\log^5{x})/M$ and where
\begin{align*}
\mathcal{E}_2&:=\sum_{\substack{(q,a)=1}}\psi_0\Bigl(\frac{q}{Q}\Bigr)
\underset{\substack{(r_1s_1,a r_2s_2)=1\\ (r_2s_2,a q r_1s_1)=1}}{\sum_{r_1,r_2\sim R}\sum_{s_1,s_2\sim S}}
\frac{c_{q,r_1,s_1}\overline{c_{q,r_2,s_2}}}{q r_1 r_2s_1s_2}\sum_{\substack{n_1,n_2\sim N\\ n_1\equiv n_2\Mod{q}\\ (n_1,n_2 q r_1)=1\\(n_2,n_1 q r_2s_2)=1\\ |n_1-n_2|\ge N/(\log{x})^C}}\alpha_{n_1}\overline{\alpha_{n_2}}\\
&\qquad\qquad \times\sum_{1\le |h|\le H_2}\hat{\psi}_0\Bigl(\frac{h M}{q r_1 r_2s_1s_2}\Bigr)e\Bigl(\frac{ah\overline{n_1r_2s_2}}{q r_1s_1}+\frac{ah\overline{n_2 q r_1s_1}}{r_2s_2}\Bigr).
\end{align*}
Now absorbing the $\psi_0(q/Q)$ factors into the coefficients $c_{q,r,s}$, we see these are precisely the sums $\widetilde{\mathcal{Z}}$ and $\mathcal{Z}$ considered in Lemma \ref{lmm:Zhang1} and Lemma \ref{lmm:Zhang2}. Thus, these lemmas give the result provided we have
\begin{align}\label{eq:8.3assume1}
Q^{7} R^{12}S^{10}< x^{4-18\eps},\qquad 
Q<N<\frac{x^{1-5\eps}}{QS^2},
\end{align}
and
\begin{align}\label{eq:8.3assume2}
N Q < x^{1-4\eps},\qquad
N Q^{5/2}(RS)^3< x^{2-4\eps},\qquad
N^2 Q (RS)<x^{2-4\eps}.
\end{align}
Recalling that $Q=Q_1 x^{o(1)}$, $R=Q_2 x^{-o(1)}$, $S=Q_3 x^{-o(1)}$, observe \eqref{eq:8.3assume1} holds by assumption. Next, the first inequality in \eqref{eq:8.3assume2} follows from the second in \eqref{eq:8.3assume1}, as $N < x^{1-5\eps}/QS^2< x^{1-4\eps}/Q$. The second inequality in \eqref{eq:8.3assume2} follows, since \eqref{eq:8.3assume1} implies $Q^{3/2}R^{3}S < x^{1-4\eps}$ and so $NQ^{5/2}(RS)^3 < (x^{1-5\eps}/QS^2)Q^{5/2}(RS)^3 = x^{1-5\eps}Q^{3/2}R^3S < x^{2-9\eps}$. The third inequality in \eqref{eq:8.3assume2} follows, since $N^2QRS < (x^{1-5\eps}/QS^2)^2QRS = x^{1-10\eps}R/QS^3 < x^{2-9\eps}$ by assumption $R < QS^3$. (Throughout we may assume $QRS\ge x^{1/2-\eps}$ or else the result follows from the Bombieri-Vinogradov theorem). This completes the proof.
\end{proof}
%
%
%
%
%
%
%
\section{Proof of Proposition \ref{prpstn:TypeIII2} (type II estimate)}\label{sec:TypeIII2}

In this section we prove Proposition \ref{prpstn:TypeIII2} using the new Zhang-style estimate from the prior section, via the type II estimates below.
We recall that $S_n$ is defined by
\[
S_n :=\mathbf{1}_{n\equiv a\Mod{q}}-\frac{1}{\phi(q)}\mathbf{1}_{(n,q)=1},
\]
where the modulus $q$ (or $qr$, $qrs$) is understood in context.
%
%
%
%
%
%
%
%
\begin{lemma}[Type II estimate away from $x^{1/2}$]\label{lmm:BasicTypeII}
Let $A>0$ and $QR\le x^{127/224-\eps}$, and let $P_1,\dots, P_J\ge x^{1/7+10\eps}$ be such that $P_1\cdots P_J\asymp x$ and
\[
x^{\eps}QR <\prod_{j\in\mathcal{J}}P_j<x^{4/7-\eps}
\]
for some subset $\mathcal{J}\subseteq\{1,\dots,J\}$. 

Then we have
\[
\sum_{q\sim Q}\sum_{\substack{r\sim R\\ (qr,a)=1}}\Bigl|\mathop{\sideset{}{^*}\sum}_{\substack{p_1,\dots,p_J\\ p_i\sim P_i\forall i}}S_{p_1\cdots p_J}\Bigr|\ll_{A}\frac{x}{(\log{x})^A}.
\]
Here $\sum^*$ indicates that the summation is restricted by $O(1)$ inequalities of the form $p_1^{\alpha_1}\cdots p_J^{\alpha_J}\le B$. The implied constant may depend on all such exponents $\alpha_i$, but none of the quantities $B$.
\end{lemma}
\begin{proof}
This is \cite[Propositon 8.12]{JM1}.
\end{proof}
%
%
%
%
%
%
%
%
%
%
%
%
%
\begin{lemma}[Type II estimate near $x^{1/2}$]\label{lmm:ZhangTypeIII}
Let $A>0$ and let $Q,R,S$ satisfy 
\begin{align}
R&<QS^3,\label{eq:zhangtype3.1}\\
QS^2 &<x^{1/2-20\eps},\\
Q^2 RS &<x^{1-20\eps},\\
Q^7R^{12}S^{10}&<x^{4-20\eps}.\label{eq:zhangtype3.2}
\end{align}
Let $P_1,\dots,P_J\ge x^{1/7+10\eps}$ be such that $P_1\cdots P_J\asymp x$ and
\[
\frac{x^{1-\eps}}{QRS}\le \prod_{j\in \mathcal{J}}P_j\le x^\eps QRS,
\]
for some subset $\mathcal{J}\subseteq\{1,\dots,J\}$.

Let $\lambda_q,\nu_q,\eta_q$ be complex sequences with $|\lambda_{q}|,|\nu_q|,|\eta_{q}|\le \tau(q)^{B_0}$. Then we have
\[
\underset{(qrs,a)=1}{\sum_{q\sim Q}\sum_{r\sim R}\sum_{s\sim S}}
\lambda_{q}\nu_{r}\eta_{s}
\mathop{\sideset{}{^*}\sum}_{\substack{p_1,\dots,p_J\\ p_i\sim P_i\,\forall i}}S_{p_1\cdots p_J} \ll_A\frac{x}{(\log{x})^A}.
\]
Here $\sum^*$ indicates that the summation is restricted by $O(1)$ inequalities of the form $p_1^{\alpha_1}\cdots p_J^{\alpha_J}\le B$.  The implied constant may depend on all such exponents $\alpha_i$, but none of the quantities $B$.
\end{lemma}
\begin{proof}
This follows quickly from Proposition \ref{prpstn:Zhang}. Indeed, by Lemma \ref{lmm:Separation}, it suffices to show that
\[
\sum_{q\sim Q}\sum_{r\sim R}\sum_{\substack{s\sim S\\ (qrs,a)=1}} \lambda_{q}\nu_{r}\eta_{s}\sum_{\substack{p_1,\dots,p_J\\ p_i\in \mathcal{I}_i\forall i}}S_{p_1\cdots p_J}\ll_{B}\frac{x}{(\log{x})^{B}}.
\]
for every $B>0$ and every choice of intervals $\mathcal{I}_1,\dots,\mathcal{I}_J$ with $\mathcal{I}_j\subset[P_j,2P_j]$. By reordering the indices, we may assume $\mathcal{J}=\{1,\dots,k\}$. We now take $N\asymp \prod_{j=1}^k P_j$, and $M\asymp \prod_{j=k+1}^{J}P_j$ and
\[
\alpha_n:=\sum_{\substack{n=p_1\cdots p_k\\ p_i\in \mathcal{I}_i}}1,\qquad \beta_m:=\sum_{\substack{m=p_{k+1}\cdots p_J\\ p_i\in \mathcal{I}_i}}1,
\]
Thus since $Q^7R^{12}S^{10}<x^{4-20\eps}$, we see that Proposition \ref{prpstn:Zhang} gives the result in the range
\[
Q x^\eps<N<\frac{x^{1-7\eps}}{QS^2}
\]
Importantly, we also have the result in the mirrored range $QS^2x^{7\eps}<N< x^{1-7\eps}/Q$, by swapping the roles of $N,M$. And since $QS^2<x^{1/2-20\eps}$, together these cover the symmetric range
\[
Qx^{7\eps}<N<\frac{x^{1-7\eps}}{Q}.
\]
Finally since $Q^2RS<x^{1-20\eps}$, this covers the desired range
\begin{align*}
\frac{x^{1-\eps}}{QRS}<N\le x^\eps QRS.
\end{align*}
\end{proof}
%
%
%
%
%
%
%
%

%
%
\begin{proof}[Proof of Proposition \ref{prpstn:TypeIII2}]
We note that $S_{p_1\cdots p_J}(p_J)$ is a weighted sum over integers $p_1\cdots p_J n\sim x$ with $P^-(p_1\dots p_J n)\ge x^{1/7+10\eps}$, and so with at most 6 prime factors. Expanding this into separate terms according to the exact number of prime factors, it suffices to show
\[
\sum_{q_1\sim Q_1}\sum_{q_2\sim Q_2}
\sum_{\substack{q_3\sim Q_3\\ (q_1q_2q_3,a)=1}}\lambda_{q_1}\nu_{q_2}\eta_{q_3}
\mathop{\sideset{}{^*}\sum}_{\substack{p_1,\dots,p_J\\ p_i\sim P_i\,\forall i}}S_{p_1\cdots p_J} \ll_A\frac{x}{(\log{x})^A}.
\]
But this now follows from Lemma \ref{lmm:BasicTypeII} and Lemma \ref{lmm:ZhangTypeIII}: Indeed, if $\prod_{\mathcal J}P_j$ or $\prod_{\overline{\mathcal J}}P_j$ lies in $[x^{\eps}Q_1Q_2Q_3^2,x^{4/7-\eps}]$ then Lemma \ref{lmm:BasicTypeII} with $(Q,R)=(Q_1Q_3^2,\,Q_2)$ gives the result, since by \eqref{eq:PropQ5.1}
\begin{align*}
QR=(Q_1Q_3^2)Q_2 < (x^{1/2+\eps})(x^{1/32-\eps}) = x^{17/32}<x^{127/224-\eps}.
\end{align*} 
And if $\prod_{\mathcal J}P_j$ or $\prod_{\overline{\mathcal J}}P_j$ lies in $[x^{1/2},x^{\eps}Q_1Q_2Q_3^2]$ then Lemma \ref{lmm:ZhangTypeIII} with $(Q,R,S)=(Q_1,Q_2Q_3,\,Q_3)$ gives the result, since by \eqref{eq:PropQ5.1} we deduce \eqref{eq:zhangtype3.1}---\eqref{eq:zhangtype3.2}. Indeed,
\begin{align*}
QS^2 = Q_1Q_3^2 &< x^{1/2-2\eps}\\
Q^2RS = Q_1^2Q_2Q_3^2 &= (Q_1Q_2)(Q_1Q_3^2)\\
&< (x^{1/2+\eps})(x^{1/2-2\eps}) < x^{1-\eps}\\
Q^7R^{12}S^{10} = Q_1^7Q_2^{12}Q_3^{22} &< Q_1^7Q_2^{23} = (Q_1Q_2)^{7}Q_2^{16}\\
& \ < (x^{1/2+\eps})^7 (x^{1/32-\eps})^{16} < x^{4-9\eps},
\end{align*}
as well as
\begin{align*}
R = Q_2Q_3 < \frac{x^{1/16}}{Q_2}Q_3 < \frac{x^{1/2}}{Q_2}Q_3 < Q_1Q_3^3 = QS^3.
\end{align*}
Here we used $x^{1/2} < Q_1Q_2Q_3^2$, since otherwise the result follows by the Bombieri--Vinogradov Theorem. This completes the proof.
\end{proof}

%
%
%
%
%
%
%
%
%
\section{Proof of Proposition \ref{prpstn:4Primes} (4 prime factors)}\label{sec:4Primes}
%
%
%
%
%
%
%
%
%

In this section we prove Proposition \ref{prpstn:4Primes}. We recall the following estimate for triple convolutions.
%
%
%
%
%
%
%
%
\begin{proposition}[Estimate for triple convolutions]\label{prpstn:TripleRough}
Let $A,B_0>0$, $K L M\asymp x$, $\min(K,L,M)>x^\eps$, $a\ne 0$ and $x^{7/10-\eps}>Q>x^{1/2}(\log{x})^{-A}$. Let $L,K$ satisfy
\begin{align*}
Q x^\eps&<KL,\\
Q K&<x^{1-2\eps},\\
KL &< \frac{x^{153/224-10\eps}}{Q^{1/7}},\\
K L^{4}&<\frac{x^{57/32-10\eps}}{Q}.
\end{align*}
Let $\eta_k,\lambda_\ell,\beta_m$ be complex sequences such that $|\eta_n|,|\lambda_n|,|\beta_n|\le \tau(n)^{B_0}$ and such that $\eta_k$ satisfies the Siegel-Walfisz condition \eqref{eq:SiegelWalfisz}, and such that $\eta_k,\lambda_\ell$ be supported on integers with all prime factors bigger than $z_0$. Let 
\[
\Delta_{\mathcal{B}}(q):=\sum_{k\sim K}\sum_{\ell\sim L}\sum_{m\sim M}\eta_k\lambda_\ell\alpha_m\Bigl(\mathbf{1}_{k \ell m\equiv a\Mod{q}}-\frac{\mathbf{1}_{(k \ell m,q)=1}}{\phi(q)}\Bigr).
\]
Then we have
\[
\sum_{\substack{q\sim Q\\ (q,a)=1}}|\Delta_{\mathcal{B}}(q)|\ll_{A,B_0} \frac{x}{(\log{x})^A}.
\]
\end{proposition}
\begin{proof}
This is \cite[Proposition 8.3]{JM1}.
\end{proof}

%
%
%
%
%
%
%
%
\begin{proof}[Proof of Proposition \ref{prpstn:4Primes} from Proposition \ref{prpstn:TypeIII2} and Proposition \ref{prpstn:TripleRough}]
This follows just as in the proof of \cite[Proposition 7.3]{JM1}, except for $(Q_1, Q_2)$ replaced by $(Q_2Q_3,\,Q_1)$, and trilinear weights $\lambda_{q_1}\nu_{q_2}\eta_{q_3}$ instead of absolute values.
\end{proof}
%
%
%
%
%
%
%
%
\section{Proof of Proposition \ref{prpstn:SieveAsymptotic} (sieve asymptotics)}\label{sec:SieveAsymptotic}
%
%
%
%
%
%
%
%
In this section we prove Proposition \ref{prpstn:SieveAsymptotic} using the following Fouvry-style and small divisor estimates.

\begin{proposition}[Fouvry-style estimate]\label{prpstn:Fouvry}
Let $A>0$ and $C=C(A)$ be sufficiently large in terms of $A$. Assume that $N,M,Q,R$ satisfy $NM\asymp x$ and
\begin{align}
x^\eps Q&<N, \label{eq:Fouv1}\\
N^{6}Q^4R^{8}&<x^{5-\eps},\\
QR^2&<x^{1-\eps}N.\label{eq:Fouv2}
\end{align}
Let $\beta_m,\alpha_n$ be complex sequences such that $|\alpha_n|,|\beta_n|\le \tau(n)^{B_0}$ and such that $\alpha_{n}$ satisfies the Siegel-Walfisz condition \eqref{eq:SiegelWalfisz} and $\alpha_n$ is supported on $n$ with all prime factors bigger than $z_0:=x^{1/(\log\log{x})^3}$. Let
\[
\Delta(q):=\sum_{m\sim M}\sum_{\substack{n\sim N}}\alpha_n\beta_m\Bigl(\mathbf{1}_{m n\equiv a \Mod{q}}-\frac{\mathbf{1}_{(m n,q)=1}}{\phi(q)}\Bigr).
\]
Then we have
\[
\mathop{\sum_{q\sim Q}\sum_{r\sim R}}\limits_{(qr,a)=1}\lambda_q\nu_r\Delta(qr)\ll_{A,B_0} \frac{x}{(\log{x})^A}.
\]
\end{proposition}
\begin{proof}
This appears as \cite[Proposition p.243]{IwaniecPomykala}, noting the condition therein $x^\eps(Q+QR^2/x) < N < x^{5/6-\eps}/(QR^2)^{2/3}$ is equivalent to \eqref{eq:Fouv1}---\eqref{eq:Fouv2}. Also see \cite[Proposition 12.1]{JM1}.
\end{proof}

%
%
%
%
%
%
%
%

\begin{proposition}[Small divisor estimate]\label{prpstn:SmallDivisor}
Let $A>0$ and $C=C(A)$ be sufficiently large in terms of $A$. Assume that $N,M,Q,R$ satisfy $NM\asymp x$ and
\begin{align*}
N^6 Q^8 R^7&<x^{4-13\eps},\\
Q R^2&<x^{1-7\eps} N.
\end{align*}
Let $\beta_m,\alpha_n$ be complex sequences such that $|\alpha_n|,|\beta_n|\le \tau(n)^{B_0}$ and such that $\alpha_{n}$ satisfies the Siegel-Walfisz condition \eqref{eq:SiegelWalfisz} and $\alpha_n$ is supported on $n$ with all prime factors bigger than $z_0:=x^{1/(\log\log{x})^3}$. Let
\[
\Delta(q):=\sum_{m\sim M}\sum_{\substack{n\sim N}}\alpha_n\beta_m\Bigl(\mathbf{1}_{mn\equiv a \Mod{q}}-\frac{\mathbf{1}_{(mn,q)=1}}{\phi(q)}\Bigr).
\]
Then we have
\[
\mathop{\sum_{q\sim Q}\sum_{r\sim R}}\limits_{(qr,a)=1}|\Delta(qr)|\ll_{A,B_0} \frac{x}{(\log{x})^A}.
\]
\end{proposition}
\begin{proof}
This is \cite[Proposition 12.2]{JM1}.
\end{proof}


By combining the two results above, we now prove a small factor type II estimate for convolutions. This is a variant of \cite[Lemma 12.3]{JM1}.

\begin{lemma}[Small Factor Type II estimate for convolutions]\label{lmm:SmallTypeII}
Let $Q_1,Q_2,Q_3$ satisfy
\begin{align}\label{eq:smalltypeII1}
Q_1Q_2 &< x^{1/2+\eps}\\
Q_1^2Q_2Q_3 &< x^{1-10\eps}\\
Q_1Q_2^{8/7} Q_3^2 &< x^{4/7-10\eps}\\
(Q_1 Q_2)^2Q_3 &< x^{29/28-10\eps}
\end{align}
Let $N$, $M$ be such that $NM\asymp x$  and 
\[
x^{\eps}<N<x^{1/7+10\eps}.
\]
Let $\beta_m,\alpha_n$ be complex sequences such that $|\alpha_n|,|\beta_n|\le \tau(n)^{B_0}$ and such that $\alpha_{n}$ satisfies the Siegel-Walfisz condition \eqref{eq:SiegelWalfisz} and $\alpha_n$ is supported on $n$ with all prime factors bigger than $z_0:=x^{1/(\log\log{x})^3}$. Let
\[
\Delta(q):=\sum_{m\sim M}\sum_{\substack{n\sim N}}\alpha_n\beta_m\Bigl(\mathbf{1}_{m n\equiv a \Mod{q}}-\frac{\mathbf{1}_{(m n,q)=1}}{\phi(q)}\Bigr).
\]
Let $\lambda_q, \nu_q, \eta_q$ be complex sequences supported on $P^-(q)\ge z_0$ with $|\lambda_{q}|,|\nu_q|, |\eta_Q|\le \tau(q)^{B_0}$. Then we have
\[
\sum_{q_1\sim Q_1}\sum_{q_2\sim Q_2}
\sum_{\substack{q_3\sim Q_3\\ (q_1q_2q_3,a)=1}}\lambda_{q_1}\nu_{q_2}\eta_{q_3}\Delta(q_1q_2q_3)\ll_{A,B_0} \frac{x}{(\log{x})^A}.
\]
\end{lemma}
%
%
%
%
%
%
%
%
\begin{proof}[Proof of Lemma \ref{lmm:SmallTypeII} from Proposition \ref{prpstn:Fouvry}]
It suffices to consider $Q_1Q_2Q_3\ge x^{1/2-\eps}$, because otherwise the result follows from the Bombieri-Vinogradov Theorem. Proposition \ref{prpstn:Fouvry} with $(Q,R)=(Q_3,Q_1Q_2)$ gives the result when $N$ lies in the range
\begin{align}\label{eq:Fouvimplies}
Q_3 x^{2\eps}=\max\Bigl(\frac{(Q_1Q_2)^2Q_3}{x^{1-\eps} },Q_3 x^{2\eps} \Bigr) < N < \frac{x^{5/6-5\eps}}{((Q_1Q_2)^2Q_3)^{2/3}}.
\end{align}
Here the max in the lower bound equals $Q_3 x^{2\eps}$ since $Q_1Q_2 < x^{1/2+\eps}$ by assumption \eqref{eq:smalltypeII1}.
Next, Proposition \ref{prpstn:SmallDivisor} with $(Q,R)=(Q_2Q_3,Q_1)$ gives the result when $N$ lies in one of the ranges
\begin{align}
\frac{Q_1^2Q_2Q_3}{x^{1-7\eps} } &< N <  \frac{x^{2/3-3\eps}}{Q_1^{7/6}(Q_2Q_3)^{8/6}} \label{eq:smallimplies}
\end{align}
The ranges \eqref{eq:Fouvimplies} and \eqref{eq:smallimplies} overlap provided
\begin{align*}
Q_3 x^\eps &< \frac{x^{2/3-3\eps}}{Q_1^{7/6}(Q_2Q_3)^{8/6}}.
\end{align*}
This holds since $Q_1Q_2^{8/7} Q_3^2< x^{4/7-5\eps}$
by assumption.

Hence the result holds in the combined ranges \eqref{eq:Fouvimplies}, \eqref{eq:smallimplies}, that is,
\begin{align}\label{eq:11.3combined}
\frac{Q_1^2Q_2Q_3}{x^{1-7\eps}  } < N < \frac{x^{5/6-5\eps}}{((Q_1Q_2)^2Q_3)^{2/3}}
\end{align}
Note the lower bound in \eqref{eq:11.3combined} implies $N>x^{\eps}\ge Q_1^2Q_2Q_3/x^{1-7\eps}$ by assumption. Moreover, the upper bound in \eqref{eq:11.3combined} implies $N<x^{1/7+10\eps}$ as desired, provided
\begin{align*}
x^{1/7+10\eps} < \frac{x^{5/6-5\eps}}{((Q_1Q_2)^2Q_3)^{2/3}}.
\end{align*}
This in turn follows from $(Q_1 Q_2)^2Q_3  < x^{29/28-10\eps}$
, which completes the proof.
\end{proof}

Using Lemma \ref{lmm:SmallTypeII}, we deduce the following consequence. This is a variant of \cite[Proposition 10.1]{JM1}. Recall that $S_n$ is defined by
\[
S_n :=\mathbf{1}_{n\equiv a\Mod{q}}-\frac{1}{\phi(q)}\mathbf{1}_{(n,q)=1}.
\]
where the modulus $q$ (or $qrs$) is understood in context.

\begin{proposition}[Consequence of small factor type II estimate]\label{prpstn:VBounds}
Let $Q, R, S$ satisfy 
\begin{align}\label{eq:VBounds}
QR &< x^{1/2+\eps}\\
Q^2RS &< x^{1-10\eps}\\
QR^{8/7} S^2 &< x^{4/7-10\eps}\\
(QR)^2S &< x^{29/28-10\eps}
\end{align}
Let $\alpha_d,\beta_e,\gamma_m$ complex sequences with $|\alpha_n|,|\beta_n|,|\gamma_n|\le \tau(n)^{B_0}$ and such that $\gamma_m$ satisfies the Siegel-Walfisz condition \eqref{eq:SiegelWalfisz}. Assume that
\[
D,E,P\in [x^\eps, x^{1/7+10\eps}]
\]
and let $M N D E P\asymp x$. Let $\lambda_q,\nu_q, \eta_q$ be complex sequences supported on $P^-(q)\ge z_0$ with $|\lambda_{q}|,|\nu_q|, |\eta_q|\le \tau(q)^{B_0}$.  Then we have for every $A>0$
\[
\sum_{q\sim Q}\sum_{r\sim R}
\sum_{\substack{s\sim S\\ (qrs,a)=1}}\lambda_{q}\nu_{r}\eta_{s}
\sum_{d\sim D}\sum_{e\sim E}\sum_{p\sim P}\alpha_{d}\beta_e\sum_{m\sim M}\gamma_m \sum^*_{\substack{n\sim N\\ P^-(d),P^-(n)\ge p}}S_{n m p d e}\ll_{A,B_0} \frac{x}{(\log{x})^A}.
\]
Here $\sum^*$ indicates that the summation is restricted by $O(1)$ inequalities of the form \newline $p^{\alpha_1}d^{\alpha_2}e^{\alpha_3}m^{\alpha_4}n^{\alpha_5}\le B$. The implied constant may depend on all such exponents $\alpha_i$, but none of the quantities $B$.
\end{proposition}
%
%
%
%
%
%
%
%

\begin{proof}[Proof of Proposition \ref{prpstn:VBounds} from Proposition \ref{prpstn:Fouvry}] 
This follows just as in the proof of \cite[Proposition 10.1]{JM1}, except with trilinear weights $\lambda_q\nu_r\eta_s$ instead of absolute values. In this case, the small factor type II convolution estimate in Lemma \ref{lmm:SmallTypeII} replaces that of \cite[Lemma 12.3]{JM1}, and consequently the range $x^\eps<N<x^{1/7+\eps}$ in Lemma \ref{lmm:SmallTypeII} leads to the result under the assumption $D,E,P\in [x^\eps, x^{1/7+10\eps}]$.
\end{proof}
%
%
%
%
%
%
%
%

We are now in a position to prove the sieve asymptotics in Proposition \ref{prpstn:SieveAsymptotic}. 

%
%
%
%
%
%
\begin{proof}[Proof of Proposition \ref{prpstn:SieveAsymptotic} from Proposition \ref{prpstn:TypeIII2} and Proposition \ref{prpstn:VBounds}]
This follows just as in the proof of \cite[Proposition 7.2]{JM1}, except for $Q_1$ replaced by $Q_1Q_2Q_3^2$, and quadrilinear weights $\lambda_{q_1}\nu_{q_2}\eta_{q_3}\mu_{q_4}$ instead of absolute values. In this case, the small factor type II estimate in Proposition \ref{prpstn:TripleDivisor} replaces that of \cite[Proposition 10.1]{JM1}, and consequently we use the cutoff $y_1:=x^{\eps}$ instead of $y_1:=x^{1-\eps}/(Q_1Q_2Q_3^2)^{15/8}$.

Note we may apply Proposition \ref{prpstn:VBounds}, since \eqref{eq:PropQ5.2} implies \eqref{eq:VBounds} with $(Q,R,S)=(Q_1,Q_2,Q_3^2)$,
\begin{align*}
QR = Q_1Q_2 & < x^{1/2+\eps},\\
Q^2RS = Q_1^2Q_2Q_3^2 &= (Q_1Q_2)(Q_1Q_3^2)\\
&< (x^{1/2+\eps})(x^{1/2-2\eps}) < x^{1-\eps},\\
Q^2R^2S = Q_1^2Q_2^2Q_3^2 &= (Q_1Q_2)(Q_1Q_3^2)Q_2 \\
&< (x^{1/2+\eps})(x^{1/2-2\eps})x^{1/32-\eps} = x^{33/32} < x^{29/28-\eps},\\
QR^{8/7}S^2 = Q_1Q_2^{8/7}Q_3^{4} &< (Q_1Q_2)Q_2^{15/7} \\
&< (x^{1/2+\eps})(x^{1/32-\eps})^{15/7} < x^{127/224} < x^{4/7-\eps}.
\end{align*}

\end{proof}

%
%
%
%
%
%
%
%
%
\section{Proof of Proposition \ref{prpstn:3Primes} (3 prime factors)}\label{sec:3Primes}
%
%
%
%
%
%
%
%
In this section we prove Proposition \ref{prpstn:3Primes}. We begin by recalling the following triple divisor function estimate.
%
%
%
%
%
%
%
%
\begin{lemma}\label{lmm:KRough}
Let $x^{2\eps} \le N_1\le N_2\le N_3$ and $x^\eps\le M$ and $Q_1,Q_2\ge 1$ be such that $Q_1Q_2\le x^{1-\eps}$, $N_1N_2N_3M\asymp x$ and
\[
\frac{M Q_1^{5/2}Q_2^3}{x^{1-15\eps}}\le N_3\le \frac{x^{2-15\eps}}{Q_1^3Q_2^2 M}.
\]
Let $\alpha_m$ be a 1-bounded complex sequence, $\mathcal{I}_j\subseteq[N_j,2N_j]$ an interval and 
\[
\Delta_{\mathcal{K}}(q):=\sum_{m\sim M}\alpha_m\sum_{\substack{n_1\in \mathcal{I}_1\\n_2\in \mathcal{I}_2\\ n_3\in \mathcal{I}_3\\ P^-(n_1n_2n_3)\ge z_0}}\Bigl(\mathbf{1}_ {m n_1n_2n_3\equiv a\Mod{q}}-\frac{\mathbf{1}_{(m n_1 n_2 n_3,q)=1}}{\phi(q)}\Bigr).
\]
Then for every $A>0$ we have
\[
\sum_{\substack{q_1\sim Q_1\\ (q_1,a)=1}}\sum_{\substack{q_2\sim Q_2\\ (q_2,a)=1}}\Bigl|\Delta_{\mathcal{K}}(q_1 q_2)\Bigr|\ll_A \frac{x}{(\log{x})^{A}}.
\]
\end{lemma}
\begin{proof}
This is \cite[Lemma 20.7]{JM1}
\end{proof}
%
%
%
%
%
%

We now establish a variant of the triple divisor estimate in \cite[Proposition 11.1]{JM1}, under the weaker constraint $M=x^\eps$

%
%

\begin{proposition}[Estimate for triple divisor function]\label{prpstn:TripleDivisor}
Let $A>0$. Let $Q,R$ satisfy
\begin{align}\label{eq:QRtriplediv}
Q^3R^2 &< x^{11/7-30\eps}, \nonumber\\
Q^{11}R^{12} &< x^{6-30\eps},\\
QR &<x^{8/15-30\eps}. \nonumber
\end{align}
Let $x^{3/7}\ge N_3\ge N_2\ge N_1\ge x^\eps=:M$ satisfy $N_1N_2N_3M\asymp x$.
Let $|\alpha_m|\le \tau(m)^{B_0}$ be a complex sequence, let $\mathcal{I}_1,\mathcal{I}_2,\mathcal{I}_3$ be intervals with $\mathcal{I}_j\subseteq[N_j,2N_j]$, and let
\[
\Delta_{\mathcal{K}}(q):=\sum_{m\sim M}\alpha_m\sum_{\substack{n_1\in \mathcal{I}_1\\n_2\in \mathcal{I}_2\\ n_3\in \mathcal{I}_3\\ P^-(n_1n_2n_3)\ge z_0}}\Bigl(\mathbf{1}_ {m n_1n_2n_3\equiv a\Mod{q}}-\frac{\mathbf{1}_{(m n_1 n_2 n_3,q)=1}}{\phi(q)}\Bigr).
\]
Then we have
\[
\sum_{\substack{q\sim Q\\ (q,a)=1}}\sum_{\substack{r\sim R\\ (r,a)=1}}|\Delta_{\mathcal{K}}(q r)|\ll_A \frac{x}{(\log{x})^{A}}.
\]
\end{proposition}
\begin{proof}
First we note that by Lemma \ref{lmm:Divisor} the set of $m$ with $|\alpha_m|\ge(\log{x})^C$ has size $\ll x(\log{x})^{O_{B_0}(1)-C}$, so by Lemma \ref{lmm:SmallSets} these terms contribute negligibly if $C=C(A,B_0)$ is large enough. Thus, by dividing through by $(\log{x})^{C}$ and considering $A+C$ in place of $A$, it suffices to show the result when $|\alpha_m|\le 1$.

Since $N_1N_2N_3\asymp x^{1-\eps}$ and $N_3\ge N_2\ge N_1$ we have $N_3 \gg x^{1/3-\eps}$.
We first apply Lemma \ref{lmm:KRough} with $M=x^{\eps}$ and $(Q_1,Q_2)=(QR,\, 1)$. This gives the result provided
\begin{equation}
\frac{Q^{5/2} R^{5/2} }{x^{1-12\eps}}<N_3<\frac{x^{2-16\eps} }{Q^3 R^3 }.\label{eq:N3Range1}
\end{equation}
Similarly, we apply Lemma \ref{lmm:KRough} with $(Q_1,Q_2)=(Q,R)$, which gives the result provided
\begin{equation}
\frac{Q^{5/2}R^{3}}{x^{1-14\eps}}<N_3<\frac{x^{2-16\eps}}{Q^3 R^2}.\label{eq:N3Range2}
\end{equation}
These ranges \eqref{eq:N3Range1} and \eqref{eq:N3Range2} overlap, provided
\begin{align*}
\frac{Q^{5/2}R^{3}}{x^{1-14\eps}} < \frac{x^{2-16\eps} }{Q^3 R^3 },
\end{align*}
which holds since $Q^{11} R^{12} < x^{6-30\eps}$. Thus the result holds in the combined range
\begin{align}
\frac{Q^{5/2} R^{5/2} }{x^{1-12\eps}}<N_3<\frac{x^{2-16\eps}}{Q^3 R^2}.
\end{align}
This covers the stated range $x^{1/3-\eps}<N_3<x^{3/7+\eps}$, since by assumption
\begin{align*}
Q^3R^2 < x^{11/7-30\eps}, \qquad QR<x^{8/15-30\eps}.
\end{align*}
This gives the result.
\end{proof}
%
%
%
%
%
%
%
%
%

\begin{proof}[Proof of Proposition \ref{prpstn:3Primes} assuming Propositions \ref{prpstn:4Primes} and \ref{prpstn:VBounds}] 
This follows just as in the proof of \cite[Proposition 7.4]{JM1}, except for `$Q_1$' replaced by $Q_2Q_3^2$, and with quadrilinear weights $\lambda_{q_1}\nu_{q_2}\eta_{q_3}\mu_{q_4}$ instead of absolute values. In this case, the triple divisor function estimate in Proposition \ref{prpstn:TripleDivisor} replaces that of \cite[Proposition 11.1]{JM1}, and consequently we use the cutoff $y_1:=x^{\eps}$ instead of $y_1:=x^{1-\eps}/(Q_1Q_2Q_3^2)^{15/8}$. Note Proposition \ref{prpstn:TripleDivisor} may be applied here, since \eqref{eq:PropQ5.4} implies \eqref{eq:QRtriplediv} with $(Q,R)=(Q_1,\,Q_2Q_3^2)$: Indeed, by \eqref{eq:PropQ5.4} we have
\begin{align*}
Q^{11}R^{12}=Q_1^{11}(Q_2Q_3^2)^{12} &< Q_1^{11} Q_2^{24} = (Q_1Q_2)^{11}Q_2^{13}\\
&< (x^{1/2+\eps})^{11} (x^{1/32-\eps})^{13} = x^{189/32-2\eps} < x^{6-30\eps},\\
Q^3R^2=Q_1^3(Q_2Q_3^2)^2 &< Q_1^3 Q_2^4 = (Q_1Q_2)^{3}Q_2\\
&< (x^{1/2+\eps})^3(x^{1/32-\eps}) < x^{49/32+2\eps} < x^{11/7-30\eps},\\
QR=Q_1Q_2Q_3^2 &< (Q_1 Q_2)Q_2\\
&< (x^{1/2+\eps})(x^{1/32-\eps}) = x^{17/32} < x^{8/15-30\eps}.
\end{align*}
(Note the optimal triple above is $(Q_1,Q_2,Q_3)=(x^{15/35+2\eps},x^{1/32-\eps},x^{1/64-2\eps})$).
\end{proof}
%
%
%
%
%
%
%
%

\section*{Acknowledgments}
The author is grateful to James Maynard and Carl Pomerance for many valuable discussions.
The author is supported by a Clarendon Scholarship at the University of Oxford.
%
%
%
%
%
%

\bibliographystyle{amsplain}

\begin{thebibliography}{9}

\bibitem{AGP} 
W.~R.~Alford, A.~Granville, C.~Pomerance.
\newblock There are infinitely many Carmichael numbers.
\newblock {\em Ann. of Math.} 139:703--722, 1994.

\bibitem{BH} 
R.~C.~Baker, G.~Harman.
\newblock Shifted primes without large prime factors.
\newblock {\em Acta Arith.} 83:331--360, 1998.

\bibitem{Balog} 
A.~Balog.
\newblock $p+a$ without large prime factors.
\newblock {\em S\'em. Th\'eorie des Nombres Bordeaux} (1983-84), expos\'e 31.

\bibitem{Bombieri}
E.~Bombieri.
\newblock On the large sieve.
\newblock {\em Mathematika}, 12:201--225, 1965.

\bibitem{BFI1}
E.~Bombieri, J.~Friedlander, and H.~Iwaniec.
\newblock Primes in arithmetic progressions to large moduli.
\newblock {\em Acta Math.}, 156(3-4):203--251, 1986.

\bibitem{BFI2}
E.~Bombieri, J.~Friedlander, and H.~Iwaniec.
\newblock Primes in arithmetic progressions to large moduli. {II}.
\newblock {\em Math. Ann.}, 277(3):361--393, 1987.

\bibitem{BFI3}
E.~Bombieri, J.~Friedlander, and H.~Iwaniec.
\newblock Primes in arithmetic progressions to large moduli. {III}.
\newblock {\em J. Amer. Math. Soc.}, 2(2):215--224, 1989.


\bibitem{CrandPomer}
R.~Crandall and C.~Pomerance.
\newblock Prime numbers: a computational perspective.
\newblock Springer-Verlag, New York, 2001.

\bibitem{DeshouillersIwaniec}
J.-M. Deshouillers and H.~Iwaniec.
\newblock Kloosterman sums and {F}ourier coefficients of cusp forms.
\newblock {\em Invent. Math.}, 70(2):219--288, 1982/83.


\bibitem{ElliottHalberstam}
P.~D. T.~A. Elliott and H.~Halberstam.
\newblock A conjecture in prime number theory.
\newblock In {\em Symposia {M}athematica, {V}ol. {IV} ({INDAM}, {R}ome,
  1968/69)}, pages 59--72. Academic Press, London, 1970.
  
  
\bibitem{Erdosnormal}
P.~Erd\H{o}s.
\newblock On the normal number of prime factors of $p-1$ and some related problems concerning Euler's $\varphi$-function.
\newblock {\em Quart. J. Math.}, Oxford Ser. 6:205--213, 1935.

\bibitem{Erdospseudo}
P.~Erd\H{o}s.
\newblock On pseudoprimes and Carmichael numbers.
\newblock {\em Publ. Math. Debrecen}, 4:201--206, 1956.


P. Erdős: On pseudoprimes and Carmichael numbers, Publ. Math. Debrecen 4 (1956), 201--206

\bibitem{Fouvry}
E.~Fouvry.
\newblock Autour du th\'{e}or\`eme de {B}ombieri-{V}inogradov.
\newblock {\em Acta Math.}, 152(3-4):219--244, 1984.


\bibitem{Fouvry2}
E.~Fouvry.
\newblock Autour du th\'{e}or\`eme de {B}ombieri-{V}inogradov. {II}.
\newblock {\em Ann. Sci. \'{E}cole Norm. Sup. (4)}, 20(4):617--640, 1987.

\bibitem{FG} 
E.~Fouvry, F.~Grupp.
\newblock On the switching principle in sieve theory.
\newblock {\em J. reine angew. Math.} 370:101--125, 1986.


\bibitem{Fried} 
J.B.~Friedlander.
\newblock Shifted primes without large prime factors.
\newblock {\em Number Theory and Applications}, Kluwer, Berlin, 393--401, 1989.

\bibitem{Granvillesmooth} 
A.~Granville.
\newblock Smooth numbers: Computational number theory and beyond.
\newblock {\em Algorithmic Number Theory} (eds. J. Buhler and P. Stevenhagen), \newblock MSRI Publications 44, 2008.


\bibitem{Harman}
G.~Harman.
\newblock {\em Prime-detecting sieves}, volume~33 of {\em London Mathematical
  Society Monographs Series}.
\newblock Princeton University Press, Princeton, NJ, 2007.

\bibitem{HarmWatt} 
G.~Harman.
\newblock Watt's mean value theorem and Carmichael numbers.
\newblock {\em Int. J. Number Theory}, 4:241--248, 2008.

\bibitem{IwaniecKowalski}
H.~Iwaniec and E.~Kowalski.
\newblock {\em Analytic number theory}, volume~53 of {\em American Mathematical
  Society Colloquium Publications}.
\newblock American Mathematical Society, Providence, RI, 2004.

\bibitem{IwaniecPomykala}
H.~Iwaniec and J.~Pomykala.
\newblock Sums and differences of quartic norms.
\newblock {\em Mathematika}, 40:233--245, 1993.
  
 
\bibitem{JM1} 
J.~Maynard.
\newblock Primes in arithmetic progressions to large moduli I: fixed residue classes.
\newblock Mem. Amer. Math. Soc., to appear.

\bibitem{JM2} 
J.~Maynard.
\newblock Primes in arithmetic progressions to large moduli II: well-factorable estimates.
\newblock Mem. Amer. Math. Soc., to appear.

\bibitem{Pomer} 
C.~Pomerance.
\newblock Popular values of Euler's function.
\newblock {\em Mathematika}, 27:84--89, 1980.

\bibitem{PomerShpar} 
C.~Pomerance and I.E.~Shparlinski.
\newblock Smooth orders and cryptographic applications.
\newblock {\em Proc. ANTS-V}, Sydney, Australia, {\em Springer Lecture Notes in Computer Science} 2369:338--348, 2002.

\bibitem{Shiu}
P.~Shiu.
\newblock A {B}run-{T}itchmarsh theorem for multiplicative functions.
\newblock {\em J. Reine Angew. Math.}, 313:161--170, 1980.

\bibitem{Vinogradov}
A.~I. Vinogradov.
\newblock The density hypothesis for {D}irichet {$L$}-series.
\newblock {\em Izv. Akad. Nauk SSSR Ser. Mat.}, 29:903--934, 1965.

\bibitem{Wool}
K.~R.~Wooldridge.
\newblock Values taken many times by Euler's phi-function.
\newblock {\em Proc. Amer. Math. Soc}, 76:229--234, 1979.

\bibitem{Zhang}
Y.~Zhang.
\newblock Bounded gaps between primes.
\newblock {\em Ann. of Math. (2)}, 179(3):1121--1174, 2014.


\end{thebibliography}

\end{document}